\documentclass[12pt]{amsart}
\usepackage{amssymb, amsmath, amsthm}
\usepackage[latin1]{inputenc}
\usepackage{graphics,graphicx}
\graphicspath{{./figs/}}

\usepackage[final]{hyperref}
\usepackage{color}
\usepackage{algorithm}
\usepackage{algorithmicx}
\usepackage{algpseudocode}
\usepackage{ulem}
\newcommand\ds\displaystyle
\usepackage[a4paper, centering]{geometry}
\usepackage{mathtools}  %weak convergence symbols
\usepackage{optidef}    %Optimizatio problems definition

\usepackage{color}
\usepackage{graphicx}
\graphicspath{{./pics/}}
\newtheorem{theorem}{Theorem}
\newtheorem{proposition}[theorem]{Proposition}
\newtheorem{lemma}[theorem]{Lemma}
\newtheorem{corollary}[theorem]{Corollary}
\theoremstyle{definition}

\theoremstyle{remark}
\newtheorem{remark}[theorem]{Remark}

\def\R{\mathbb{R}}

\def\N{\mathbb{N}}

\def\Z{\mathbb{Z}}

\def\Lra{\Longrightarrow}

\definecolor{verde}{RGB}{20,150,100}

\newcommand{\Om}{\Omega}
\newcommand{\om}{\omega}
\newcommand{\vphi}{\varphi}

\newcommand{\ra}{\rightarrow}

\def \d{\delta}

\newcommand{\vps}{\varepsilon}

\newcommand{\sm}{\setminus}

\newcommand{\rau}{\rightharpoonup}

\newcommand{\sq}{\subseteq}
\newcommand{\ov}{\overline}
\newcommand{\nif}{{n \rightarrow +\infty}}

\newcommand{\B}{\mathbf{B}}

\renewcommand{\L}{\mathbf{L}}

\newcommand{\1}{\mathbf{1}}
\renewcommand{\div}{\text{div}\,}
\renewcommand{\vec}[1]{\mathbf{#1}}
\renewcommand{\d}{\:\mathrm{d}}

\newcommand{\wlim}[1]{\xrightharpoonup{#1}}

\begin{document}

%%%%%%%%%%%%%%%%%%%%%%%%%%%%%%%%%%%%%%%%%%%%%%%%%%%%%%%%%%
 %amsart format
\title[]{Maximization of Neumann eigenvalues}

\author[D. Bucur]
{Dorin Bucur}
\address[Dorin Bucur]{Univ. Savoie Mont Blanc, CNRS, LAMA \\
73000 Chamb\'ery, France
}
\email[D. Bucur]{dorin.bucur@univ-savoie.fr}
\author[E. Martinet]
{Eloi Martinet}
\address[Eloi Martinet]{Univ. Savoie Mont Blanc, CNRS, LAMA \\
73000 Chamb\'ery, France}
\email[E. Martinet]{eloi.martinet@univ-smb.fr}
\author[E. Oudet]
{Edouard Oudet}
\address[Edouard Oudet]{Univ. Grenoble Alpes, CNRS, LJK, 38041 Grenoble, France}
\email[E. Oudet]{edouard.oudet@univ-grenoble-alpes.fr}
\thanks{The authors were supported by the LabEx PER\-SYVAL-Lab GeoSpec (ANR-11-LABX-0025-01) and ANR SHAPO (ANR-18-CE40-0013).
}

\subjclass[2010]{}

\date{\today}
\maketitle

\begin{abstract}

This paper is motivated by the maximization of  the $k$-th eigenvalue of the Laplace operator with Neumann boundary conditions among domains of ${\mathbb R}^N$ with prescribed measure.
We relax the problem to the class of (possibly degenerate) densities in $\R^N$ with prescribed mass and prove the existence of an optimal density.
For $k=1,2$ the two problems are equivalent and the maximizers are known to be one and two equal balls, respectively. For $k \ge 3$ this question remains open, except in one dimension of the space where we prove that the maximal densities correspond to a union of $k$ equal segments.   This result provides sharp upper bounds for Sturm-Liouville eigenvalues and proves the validity of the P\'olya conjecture in the class of  densities in $\R$.
Based on the relaxed formulation, we provide numerical approximations of optimal densities for $k=1, \dots, 8$ in $\R^2$.

\end{abstract}
\tableofcontents

\section{Introduction}
Shape optimization problems involving   partial differential equations with Neumann  conditions on the free boundary  appear naturally in different mathematical models coming from structural mechanics, image analysis, biology, etc. The mathematical analysis of such problems is in general complicated, mainly because of the absence of any interface energy which controls the interplay between the PDE and the geometry.   Indeed, contrary to Dirichlet or Robin boundary conditions which implicitly involve a boundary energy, the freedom of a Neumann interface raises crucial difficulties. In some situations, for instance crack propagation models or Mumford-Shah functional, a boundary energy, typically the Hausdorff measure (which is not related to the PDE itself) is naturally present. However, in some other situations, like those involving the  vibration of free structures, such extra energy it is not natural on the free parts.

We focus in this paper on a class of problems which involve the spectrum of the Laplace operator with Neumann boundary conditions.
Let $N \ge 1$ and $\Om \sq \R^N$ be an open, bounded, Lipschitz set. The Laplace operator with Neumann boundary conditions has a spectrum consisting on  eigenvalues,  denoted (counting their multiplicities)
$$0 = \mu_0(\Om)< \mu _1(\Om) \le \mu_2(\Om) \le \dots \ra +\infty.$$
For every $k \ge 1$, we have
$$\mu_k(\Om) = \min_{S\in{\mathcal S}_{k+1}} \max_{u \in S\sm \{0\}} \frac{\int_\Om |\nabla u|^2 dx}{\int_\Om u^2 dx},$$
where ${\mathcal S}_k$ is the family of all subspaces of dimension $k$ in
$H^1(\Om)$. Then, for some $u \in H^1(\Om)\sm \{0\}$
$$
\begin{cases}
-\Delta u = \mu_k(\Om) u \mbox { in } \Om,\\
\frac{\partial u}{\partial n} = 0  \mbox { on } \partial \Om,
\end{cases}
$$
in a weak sense (which is also strong as soon as $\Om$ is of class $C^2$).

We are concerned with the  following shape optimization problem: given $m>0$, solve
\begin{equation} \label{bmo02}
\max \{ \mu_k(\Om) :  \Om \sq \R^N, \Om \mbox{ bounded, open and Lipschitz }, |\Om| =m\}.
\end{equation}
This question is, in general, open. For the Laplace operator with Dirichlet boundary conditions, the similar (minimizing) question has been intensively investigated in last 30 years. Since the seminal existence result by Buttazzo and Dal Maso \cite{BDM93}, a series of results,  both of analytical and numerical type, have been obtained  (see the recent survey  \cite{He17}). In the Dirichlet case, the  description of relaxed problems is known. This is due to full understanding of  the Gamma convergence limits for the energy functionals (and hence of the spectrum), while local analysis by free boundary techniques leads to qualitative information on the optimal shapes. If Dirichlet boundary conditions are replaced by Neumann conditions, as in problem \eqref{bmo02}, several deep, new difficulties appear, changing completely the nature of the problem.

First, on nonsmooth domains the resolvent of the Neumann Laplacian is not necessarily compact, so that the spectrum may not consist on eigenvalues. Keeping track of the Lipschitz character of a maximizing sequence is an impassable challenge.
Second, there is no Gamma convergence description of limits of the energy functionals, and even if such a result was available, it would not be enough for the control of the spectrum, because of the absence of collective compactness of Sobolev spaces $H^1(\Om)$ in $L^2(\R^N)$. Third, local analysis to search some regularity of the free boundary seems out of reach, since problem \eqref{bmo02} is  of $\max$-$\min$-$\max$  type; consequently, one can not test maximality by local perturbations of the eigenfunctions. Fourth, the absence of any interface energy makes the final answer different from the case of Dirichlet (or Robin) boundary conditions, being unclear: for some values of $k$ optimal shapes do exist, while for others, this may not be the case.

Another point of interest in problem \eqref{bmo02} is related to the long standing P\'olya conjecture, which states that
$$\forall k \in \N,\;\;\;\; \mu_k(\Om) \le \frac{4\pi^2 k^\frac2N}{(\om_N|\Om|)^\frac2N}.$$
This inequality is known to hold  only  for some particular domains in $\R^N$, like those tiling the space  \cite{La97} or with a particular geometric structure \cite{FrSa22}. In general, for arbitrary domains, it has only been proved for  $k=1,2$. Any qualitative information on the solution to problem \eqref{bmo02} may provide useful information on the  P\'olya conjecture.

Coming back to Problem \eqref{bmo02}, let us briefly recall the known results. For $k=1$ the solution of problem \eqref{bmo02}   corresponds to a ball of volume $m$. This was proved by Szeg\"{o} \cite{Sz54} (for smooth simply connected domains in $\R^2$) and by  Weinberger \cite{We56} (for Lipschitz domains in $\R^N$). For $k=2$ the solution is the union of two disjoint equal balls. This result was proved in $\R^2$ for smooth simply connected sets by Girouard, Nadirashvili and Polterovich in \cite{GNP09} and for general domains in $\R^N$ by Bucur and Henrot in \cite{BuHe19}. As a consequence,  the P\'olya conjecture holds for $k=1,2$. In this direction, for $k\ge 3$  Kr\" oger found in  \cite{Kr92} a series of  bounds  larger than the conjectured ones, however, respecting the growth of the Weyl law.

 Although it puzzled the spectral geometry community in the last years, problem \eqref{bmo02} remains largely open for $k\ge 3$ (and $N\ge 2$). We can only refer to several numerical approximations of suspected geometries which maximize $\mu_k$, see for instance \cite{AO17,AF12, Be15}, but the optimal geometries are not even proved to exist! The case $N=1$ is trivial and the answer is the union of $k$ segments of length $\frac mk$ (possibly joining at their extremities).

\medskip
The main observation motivating this work comes from \cite{BuHe19}. Precisely, there it is proved for $k=1,2$ that the optimal geometries (a ball, two equal balls, respectively) are in fact optimal in the larger class of (possibly degenerate) densities satisfying a mass contraint. In this class, a Lipschitz set is identified with the density equal to its characteristic function. Out of this observation, a natural question emerges: is this phenomenon true for every $k$? A positive answer would open the way to the proof of existence of optimal sets in \eqref{bmo02} for every $k$, while a negative answer would even raise more questions.

\medskip
Let $\rho : \R^N \ra [0,1]$ be a measurable function such that $0<\int_{\R^N} \rho dx <+\infty$. We consider  the {\it possibly degenerate} eigenvalue problem defined via the relaxation of the Rayleigh quotient: for every integer $k\ge 0$, we set
\begin{equation}\label{bmo03.0}
 \mu _k(\rho) := \inf_{S\in{\mathcal S}_{k+1}} \max_{u \in S} \frac{\int_{\R^N} \rho|\nabla u|^2 dx}{\int_{\R^N} \rho u^2 dx},\end{equation}
where ${\mathcal S}_{k+1}$ is the family of all subspaces of dimension $k+1$ in
\begin{equation}\label{bmo03}
\{u\cdot 1_{\{\rho (x)>0\}}: u \in C^\infty_c (\R^N)\}.
\end{equation}
Clearly, if $\rho$ satisfies some suitable assumptions (for instance if it equals the characteristic function of a bounded, open, Lipschitz set or if it is a Gaussian measure in $\R^N$, etc.), the Rayleigh quotient  above leads to a classical spectrum associated to a positive, self-adjoint, compact operator. If $\rho$ is just arbitrary, the definition \eqref{bmo03.0} itself is correct in the sense that the numbers $\mu_k(\rho)$ are well defined, but there is no an interpretation in terms of the spectrum of a well defined operator. By abuse of language, we still call them {\it eigenvalues of the density} $\rho$.

It is natural to consider the (relaxed) problem
\begin{equation} \label{bmo01}
\sup \{ \mu_k(\rho) :  \rho : \R^N \ra [0,1], \int_{\R^N}\rho dx =m\},
\end{equation}
or its scale invariant version
\begin{equation} \label{bmo01.1}
\mu_k^*:= \sup \{ \Big (\int _{\R^N} \rho dx \Big)^\frac 2N\mu_k(\rho) :  \rho : \R^N \ra [0,1]\},
\end{equation}
and to observe that, in general, this value is not smaller than the one given by problem \eqref{bmo02}, from the simple fact that the class of Lipschitz sets is implicitly contained in the class of  densities.
In \cite{BuHe19} it was proved that for $k=1,2$ problems \eqref{bmo02} and \eqref{bmo01} are indeed equivalent, the maximizers being the same, corresponding to one and two equal balls, respectively.

\medskip
This observation raises several questions.

 \medskip
\noindent{\it
Does problem \eqref{bmo01.1}  have a solution?}
 Our first result is that
 problem  \eqref{bmo01.1}  has a solution in $\R^N$ possibly consisting in a {\it collection} of  at most $k$ densities.
  \begin{theorem}\label{bmo06.1}
The maximal value $\mu_k^*$  in \eqref{bmo01.1} is attained. Precisely,   there exist $j \in \N$, $j \le k$, $\rho_1, \dots, \rho_j : \R^N \ra [0,1]$ and $n_1, \dots, n_j \in \N$ with $n_1+\dots + n_j= k+1-j$ such that
$$\sum_{i=1}^j \int_{\R^N} \rho_i dx = 1 \quad \mbox {and} \quad \mu_k^*= \mu_{n_1}(\rho_1) = \dots = \mu_{n_j}(\rho_j).$$
\end{theorem}
Related to the distribution of eigenvalues, the notion  of collection of densities in Theorem \ref{bmo06.1} plays a similar role as the connected components of an open set.
 We also point out that the existence result above extends to general functionals of eigenvalues
 $$\rho \to F(\mu_1(\rho), \dots, \mu_k(\rho)),$$
 where $F$ is nondecreasing in each variable and upper semicontinuous.
The key technical point in the proof of Theorem \ref{bmo06.1} is to get some control of the concentration of mass for a maximizing sequence of densities; this is a difficult task because of the absence of any interface energy. In general, we are not able to prove or disprove the nondegeneracy of optimal densities.

 Following \cite{BuHe19},  for $k=1,2$ problems \eqref{bmo02} and \eqref{bmo01} are equivalent. For $k \ge 3$, not only this is a highly challenging question which requires regularity analysis of optimal densities, but the answer might be negative, as our numerical results suggest. To this question, we give a complete answer  in  dimension one of the space, i.e. for $N=1$, where we prove that an optimal density corresponds to a union of equal segments of lengths $\frac mk$, possibly touching at their extremities.
 \begin{theorem}[One dimension of the space]\label{bmo20.1}
In $\R$, $\forall k \in \N$
$$\mu_k^*= \pi^2k^2.$$
Equality is attained for $\rho$ being the characteristic function associated to the union of at most $k$ open, pairwise disjoint segments of total length equal to $m:= \int_\R\rho dx$, each one with length an entire multiple of  $\frac mk$.

\end{theorem}
The argument  requires fine arguments from topological degree theory and, in fact,  leads to a slightly stronger inequality (see Lemma \ref{bmo15.2}). In particular,  Theorem \ref{bmo20.1} provides  sharp upper bounds for Sturm-Liouville eigenvalues (see  \cite{Kr99} for a related problem involving concave density functions).

\medskip
\noindent{\it Does the P\'olya conjecture hold for densities?} A second consequence of   Theorem \ref{bmo20.1}  is the validity of the the P\`olya conjecture for densities. Precisely, for every  $\rho \in \L^1(\R, [0,1])$ we have
$$\forall k \in \N, \; \mu_k(\rho) \le \frac{\pi^2k^2}{\Big (\int_\R \rho dx \Big)^2 }.$$

 A natural question is whether  for every  $N\ge 2$ the P\'olya conjecture holds for densities in $\R^N$
$$\forall k \in \N, \mu_k(\rho) \le \frac{4\pi^2 k^\frac2N}{(\om_N\int_{R^N}\rho dx )^\frac2N} ?$$

While for $k=1,2$ the inequality is true, we prove that the estimates obtained by Kr\"oger \cite{Kr92} in the classical setting,  continue to hold for densities.
This result brings support to the thesis that the
P\'olya conjecture might hold for densities. With respect to the classical setting, the main improvement is that an optimal density does exist for each $k$!
\begin{theorem}[Kr\" oger estimates for densities] \label{bmo33.1}
Let $N \ge 2$, $\rho \in \L^1(\R^N) \cap \L^\infty(\R^N, \R^+)$, $\rho \not \equiv 0$. Then
\[\forall k \in \N, \;\;\; \mu_k(\rho) \leq 4\pi^2\Bigg( \frac{(N+2)k}{2\omega_{N}}\frac{||\rho||_\infty}{||\rho||_1} \Bigg)^{2/N},\]
where $\omega_{N}$ is the volume  of the unit ball of $\R^N$.
\end{theorem}

\medskip
\noindent{\it What is the geometry of the optimal densities?} We build up  a numerical approach and give approximations of optimal densities in $\R^2$, for $k=1, \dots, 8$. Contrary to the shape optimization problem \eqref{bmo02}, our numerical procedure is justified and, formally, leads to an approximation of the solution. A natural question is to compare our optimal densities with the previous results for optimal shapes from \cite{AO17,AF12, Be15}. For $k=1,2$  we recover the theoretical results proved in \cite{BuHe19}, while for $k=3, \dots, 8$ the numerical values for the optimal densities are quite close to the optimal values for domains from  \cite{AO17,AF12, Be15}, being, as expected,  larger.  Moreover, the densities are close to be characteristic functions. However, some surprising non simply connected geometries of their level sets  are observed for $k=5$ and $k=8$, suggesting that the numerical optimal shapes previously known could be improved. Although we can not rigorously justify that  the optimal densities do not correspond to  domains for $k=3, \dots, 8$, the numerical computations seem to suggest this fact.

\medskip
The paper is organized as follows. In Section \ref{bmo:s2}  we prove Theorem \ref{bmo06.1}, Section \ref{bmo:s3} contains the analysis of the one dimensional case with the proof of Theorem \ref{bmo20.1}, Section \ref{bmo:s4} contains the discussion about the P\`olya conjecture and the proof of Theorem \ref{bmo33.1} while Section 5 is dedicated to the numerical computations.

\section{Existence result: proof of Theorem \ref{bmo06.1}} \label{bmo:s2}

For every $m>0$ and open set $D \sq \R^N$, let us denote
$$\L^1(D, [0,1]):= \{  \rho : D \ra [0,1]: \rho \in L^1(D)\}$$
$$ \L^1_m(D, [0,1]):= \{  \rho \in  \L^1(D, [0,1]): \int_{D}\rho dx =m\}.$$

In order to analyse existence of a solution to Problem \eqref{bmo01.1} in $\L^1(\R^N, [0,1])$ the main technical difficulty is related to handling the behavior of densities with unbounded support. For instance, if $(\rho_n)_n$ is a maximizing sequence such that, after possible translations, there exists a ball which contains all supports of $\rho_n$, then the existence of an optimal density would follow immediately as a consequence of an upper semicontinuity result for $L^\infty $ weak-$*$ convergence of densities (Lemma \ref{bmo15} below) and the preservation of the constraint at the limit.  Nevertheless, a maximizing sequence could, a priori, have densities with very large, or unbounded support. In this case, the upper semicontinuity result still works, but it is not enough to prove existence. In fact, there is no guarantee that a limit density does satisfy the constraint since the constant function $1$ does not belong to $L^1(\R^N)$. In other words, to prove that the constraint is satisfied by a limit density one should gather qualitative information  on the maximizing sequence itself. Our idea is to prove that a maximizing sequence enjoys some mass concentration properties around at most $k$ spots.  Contrary to the case of Dirichlet boundary conditions we can not perform a surgery of the domains in order to remove possibly insignificant parts of a density (for instance long and thin tails of their support). This is a consequence of the max-min-max structure of Problem \eqref{bmo01.1}.

There is a second issue related to the behavior of the spectrum for disconnected sets. The maximizer of $\mu_2$ is known to be the union of two equal balls. The spectrum of the union of balls is the union of spectra of each ball, multiplicity being counted.
 In case of densities, we have to mimic this behaviour in terms of a {\it collection} of densities, since the notion of connectedness is somehow unclear.

 We begin with  technical upper semicontinuity result. Below, by convention, if $\rho=0$ then $\forall k \ge 0, \mu_k(\rho ) =+\infty$.
\begin{lemma}\label{bmo15.1}
Assume $\rho ,\rho _n \  \in L^1(\R^N, [0,1]) $ satisfy $\rho _n \rau \rho $ weak-* in $L^\infty (\R^N)$. Then
$$\forall k \ge 1, \;\mu_k(\rho ) \ge \limsup_\nif \mu_k(\rho_ n ).$$
\end{lemma}
\begin{proof}
The proof of the result is standard.   Assume $\rho  \not= 0$, otherwise the inequality is trivially true. Let $\vps >0$ be fixed and let
$S= \operatorname{span}\{u_01_{\{\rho  >0\}} , \dots, u_k1_{\{\rho  >0\}} \}$ with $ u_0, \dots, u_k  \in C^\infty _c (\R^N)$ be  {an admissible subspace for the computation of $\mu_k(\rho )$} such that
$$\mu_k(\rho ) \ge \max _{u\in S \setminus\{0\}} \frac{ \int_{\R^N} |\nabla u|^2 \rho  dx}{ \int_{\R^N} u^2 \rho  dx}-\vps.$$
For each index $n$, assume that
$$
u_n:=\sum_{i=0}^k \alpha_i^n u_i
$$
attains the maximum of
$$
\max _{u\in S^n \setminus\{0\}} \frac{ \int_{\R^N} |\nabla u|^2 \rho_n dx}{ \int_{\R^N} u^2 \rho_n dx},
$$
where    $S^n = \operatorname{span}\{u_0 1_{\rho_n >0} , \dots, u_k 1_{\rho_n >0} \}$.   Note that for $n$ large, this space is   of dimension $k+1$ so that $\mu_k(\rho_n) \le \max _{u\in S^n \setminus\{0\}} \frac{ \int_{\R^N} |\nabla u|^2 \rho_n dx}{ \int_{\R^N} u^2 \rho_n dx}$. Without restricting the generality, we may assume that
$$
\sum_{i=0}^k (\alpha_i^n)^2=1, \qquad \;\alpha_i^n \ra \alpha_i.
$$
Denoting $ \tilde u:=\sum_{i=0}^k \alpha_i u_i$, we have

\begin{equation}
\label{bmo16}
\lim_\nif \int_{\R^N} |\nabla u_n|^2 \rho_n dx =  \int_{\R^N} |\nabla \tilde u|^2 \rho dx\qquad\text{and}\qquad
 \lim_\nif \int_{\R^N} u_n^2 \rho_n dx = \int_{\R^N}  \tilde u^2 \rho dx,
 \end{equation}
so that
$$\mu_k(\rho ) + \vps \ge \limsup_\nif \mu_k(\rho_n).$$
Taking $\vps \ra 0$, we conclude the proof.
\end{proof}

\smallskip
\noindent{\bf Collection of densities.}
Given the non-zero densities $\rho_1, \rho_2, \dots, \rho_j$, we formally denote their collection as
$$\ov \rho = \rho_1\sqcup \rho_2\sqcup  \dots\sqcup  \rho_j$$
and define the eigenvalues of  the  collection,  as follows. We consider the family
of eigenvalues of each density, take their union
$$\cup_{i=1}^j \{\mu_k(\rho_i) : k\ge 0\},$$
keep track of multiplicity and relabel them $\mu_k(\ov \rho)$ in increasing order. It can be easily noticed that $\mu_0(\ov \rho)= \mu_1(\ov \rho)= \dots =\mu_{j-1}(\ov \rho)=0$.

Of course, if the supports of $\rho_i$ are all bounded, then we can build a density $\rho $ in $\R^N$ such that $\mu_k(\rho)= \mu_k(\ov \rho)$ by just translating the supports of each density, to place them pairwise at strictly positive distance. If at least one support is unbounded, we are not, in general, able to build such a density.

Let $\ov \rho = \rho_1\sqcup \rho_2\sqcup \dots \sqcup \rho_j$ and let $\ov \rho '=\rho_1\sqcup \rho_2\sqcup \dots \sqcup \rho_{j-1}$. Then, from the definition,
$$\forall k \ge 1, \mu_k(\ov \rho) \le  \mu_k(\ov \rho').$$
In other words, if we drop a component of $\ov \rho$, the eigenvalues can not decrease. We have the following.

\begin{lemma}\label{bmo17}
Assume $\rho_n \in \L^1 (\R^N,[0,1]) $ is such that $\rho_n = \rho_n^0+\rho_n^1+\dots +\rho_n^{j}$, $\int_{\R^N}\rho_n^0 dx \ra 0$, and  for every $l =1, \dots, j$, for some sequences $(y_n^l)_n$ we have $\rho_n^l(y_n^l +\cdot)\rau \rho^l$ weak-* in $L^\infty (\R^N)$. Assume that for all $1 \le l \not =h\le j$, $dist(\{\rho_n^l>0\} , \{\rho_n^h>0\}) \ra + \infty$. Then, denoting
$\ov \rho =\rho_1\sqcup \rho_2\sqcup \dots \sqcup \rho_j$, we have
$$\forall k \ge 1, \;\mu_k(\ov \rho) \ge \limsup_\nif \mu_k(\rho_n).$$
\end{lemma}
\begin{proof}
By definition, there exists $k_1, \dots, k_j$ such that $k_1+\cdots+ k_j +(j-1)=k$ and
$$\mu_k (\ov \rho)= \mu_{k_1}(\rho_1),$$
$$\forall l=1, \dots, j, \, \mu_{k_l+1}(\rho_l) \ge \mu_k(\ov \rho)\ge \mu_{k_l}(\rho_l).$$
We follow the same arguments as in Lemma \ref{bmo15.1}. For some $\vps >0$, we choose
$\forall l=1, \dots, j$,   $u^l_0,\dots, u^l_{k_l}$, such that $S^l_{k_l+1}= \operatorname{span}\{u^l_01_{\rho_l >0} , \dots, u^l_{k_l}1_{\rho_l >0} \}$ is of dimension $k_l+1$ and satisfies
$$\mu_{k_l}(\rho_l) \ge \max _{u\in S^l_{k_l+1}\setminus\{0\}} \frac{ \int_{\R^N} |\nabla u|^2 \rho_l dx}{ \int_{\R^N} u^2 \rho_l dx}-\vps.$$
For each $l=1, \dots, j$, for each index $n$,
we consider the set of test functions $\{u_i^l (\cdot-y_n^l): l=1, \dots, j, i=0, \dots, k_l\}$. For $n$ large enough, the dimensions of $S^{n,l}_{k_l+1}= \operatorname{span}\{u^l_0 (\cdot-y_n^l)1_{\rho^l_n >0} , \dots, u^l_{k_l} (\cdot-y_n^l)1_{\rho^l_n >0} \}$ equal ${k_l+1}$. Since the supports of $u^l_i$ are bounded and $dist(\{\rho_n^l>0\} , \{\rho_n^h>0\}) \ra + \infty$, the space
$$S^n=\operatorname{span}\{u_i^l (\cdot-y_n^l)1_{\rho_n>0}: l=1, \dots, j, i=0, \dots, k_l\}$$
is of dimension $k+1$ and  for every $l\not= s$ the supports of $u^l_i (\cdot-y_n^l)$ and $u^s_r(\cdot-y_n^s)$ are disjoint so that

$$\mu_k(\rho_n) \le \max _{u \in S^n\setminus\{0\}} \frac{ \int_{\R^N} |\nabla u|^2 \rho_n dx}{ \int_{\R^N} u^2 \rho_n dx}.
$$
Following the same arguments as in Lemma \ref{bmo15.1} applied for every $l=1, \dots, j$, we conclude the proof.

\end{proof}

In order to prove Theorem \ref{bmo06.1}, we  formulate the following problem
\begin{equation} \label{bmo05}
\max \{ \mu_k(\ov \rho) :  \ov \rho = \rho_1\sqcup \rho_2\sqcup \dots \sqcup \rho_k, \rho_j : \R^N \ra [0,1], \sum _{j=1}^k \int_{\R^N}\rho_j dx =m\}.
\end{equation}
Of course, a solution may have $k-1$ vanishing densities, in which case their eigenvalues equal $+\infty$ and they do not contribute to the computation of $\mu_k(\ov \rho)$.   Roughly speaking, the number of nonzero densities in the maximization of $\mu_k$ can not exceed $k$, otherwise $\mu_k$ equals to $0$. The number of nonvanishing densities mimics the number of the connected components of an optimal set.

\noindent{\bf Proof of Theorem \ref{bmo06.1}.}
The case $k=1,2$ is known from \cite{BuHe19}. Let fix $k\ge 3$ and note that if we consider two different values of $m$, the solutions will be the same, up to some rescaling.  In order to prove Theorem \ref{bmo06.1}, we start with several technical results. We recall the following result from \cite[Corollary 3.12]{GNY2004}.

\medskip
\begin{lemma}[Grigor'yan-Netrusov-Yau]
 \label{bmo07}
 Let $ \rho \in \L^1_m(\R^N, [0,1])$. There exists a dimensional constant $c_N$ and $k$ annuli $(A_{x_i, r_i, R_i})_{i=1,k}$ such that
 $$\forall i=1, \dots, k, \; \int_{A_{x_i, r_i, R_i}} \rho dx \ge \frac{c_Nm}{k},$$
 $$(A_{x_i, \frac {r_i}{2}, 2R_i})_{i=1,k} \mbox{ are pairwise disjoint}.$$
 \end{lemma}
 Below, we give a first result relating the value of the $k$-th eigenvalue of $\rho$ to its concentration of mass.
 \begin{lemma}[Geometric control of the spectrum]
 \label{bmo08}
  Let $\rho\in \L^1_m(\R^N, [0,1])$   such that $\mu_k(\rho) >0$. There exists a ball $B_{x, R^*}$ with
 $$R^*=\sqrt{\frac{4(k+1)}{c_N \mu_k(\rho)}}$$
 such that
 $$\int_{B_{x,R^*}}\rho dx \ge \frac{c_Nm}{k+1}.$$
\end{lemma}
The meaning of this result is the following: if $\mu_k(\rho)$ is large (as we expect it as a maximizer), then an important fraction of the mass of $\rho$ concentrates on a ball of small, controlled radius $R^*$.
\begin{proof}
We start by applying Lemma \ref{bmo07} for $k+1$ and get the   annuli $(A_{x_i, r_i, R_i})_{i=1,k+1}$ such that $(A_{x_i, \frac {r_i}{2}, 2R_i})_{i=1,k+1} $ are pairwise disjoint and
 $$\forall i=1, \dots, k+1, \; \int_{A_{x_i, r_i, R_i}} \rho dx \ge \frac{c_Nm}{k+1}.$$
 We perform the following transformation of the annuli: if $A_{x_i, r_i, R_i}$ does not contain another annulus inside (i.e. inside the ball $B_{x_i,r_i}$) we fill it and replace it with the ball $B_{x_i, R_i}$ (roughly speaking corresponding to $r_i=0$). The statement of Lemma \ref{bmo07} continues to be valid. From now on, we work with this new family of annuli.

 We claim that for every $r_i>0$ there exists some $R_j$ such that
 \begin{equation}
\label{bmo11}
 R_j < r_i.
 \end{equation}
 Indeed, since $r_i >0$, in view of our transformation above, inside the ball $B_{x_i, r_i}$ there should be another annulus $A_{x_j, r_j, R_j}$, so that $R_j <r_i$.

 Let us denote
 $$R^* =\min \{r_i, r_i >0\} \cup \{R_j: j=1, \dots, k+1\}.$$
 In view of the previous observation, $R^*$ equals some $R_j$.
 We build the following test functions.
 \begin{itemize}
 \item On an annulus $A_{x_i, \frac {r_i}{2}, 2R_i}$
 \begin{equation}
\label{bmo09}
\vphi_i(x)=
\begin{cases}
1&\hbox{if }x \in A_{x_i, r_i, R_i}\\
\frac{d(x,B_{x_i, \frac{r_i}{2}})}{\frac{r_i}{2}}&\hbox{if  }x \in A_{x_i, \frac{r_i}{2}, r_i}\\
\frac{d(x,B^c_{x_i, 2R_i})}{R_i}&\hbox{if  }x \in A_{x_i, R_i, 2R_i}\\
0&\mbox{elsewhere}.
\end{cases}
\end{equation}
\item On a ball $B_{x_i, R_i}$
 \begin{equation}
\label{bmo10}
\vphi_i(x)=
\begin{cases}
1&\hbox{if }x \in B_{x_i, R_i}\\
\frac{d(x,B^c_{x_i, 2R_i})}{R_i}&\hbox{if  }x \in A_{x_i, R_i, 2R_i}\\
0&\mbox{elsewhere}.
\end{cases}
\end{equation}
 \end{itemize}
 Then for every $\vphi$ defined above we have
 $$\forall x \in \R^N, |\nabla \vphi (x)|\le \frac{2}{R^*},$$
 $$\int |\nabla \vphi |^2 \rho dx \le \frac{4m}{(R^*)^2},$$
 $$\int   \vphi ^2 \rho dx \ge \frac{c_Nm}{k+1}.$$
 Since all functions $\vphi_i$ have pairwise disjoint support, we can use them as test for $\mu_k$ and get
 $$\mu_k(\rho) \le \frac{\frac{4m}{(R^*)^2}}{\frac{c_Nm}{k+1}}.$$
 This implies
 $$R^*\le \sqrt{\frac{4(k+1)}{c_N \mu_k(\rho)}}.$$
 In view of \eqref{bmo11}, $R^*$ is attained by some $R_j$, so that the ball $B_{x_j, R_j}$ satisfies the conclusion of the lemma.
\end{proof}
\begin{lemma}[Enhanced geometric control of the spectrum]
 \label{bmo12}
 Let $\rho \in \L^1_m (\R^N, [0,1])$  such that $\rho=\rho_0+\rho_1+\dots +\rho_{j}$ and let $R>0$. Assume that
 $$\forall 1\le l\not= i\le j, \;\; dist (\{\rho_l >0\}, \{\rho_i >0\}) \ge 3R.$$
  Moreover, assume $\forall 1\le l \le j$, $m_l=\int \rho_l dx >0$ and denote $m_0= \int \rho_0 dx$.

 Then, for every $l\in 1, \dots, j$, there exists $R_l^*>0$ and $ x_l\in \R^N$ satisfying
  \begin{equation}
\label{bmo13}
\frac{1}{R_l^*} \ge \Big [ \frac 12 \Big ( \frac{\mu_k(\rho) c_N m_l}{(k+1)(m_l+m_0)} \Big )^\frac12-\frac{1}{2R} \Big] ^+
 \end{equation}
 and
 $$\int_{B_{x_l, R_l^*}} \rho_l \ge \frac{c_Nm_l}{k+1}.$$
 \end{lemma}
 This lemma gives a control on the concentration of masses in a dichotomy situation. If $\mu_k(\rho)$ is not small, then on every large region where there is some positive mass of $\rho$, there should also be some concentration of this mass on a ball with controlled radius, the control being in terms of $\mu_k(\rho)$.
\begin{proof}
We rely  again on Lemma \ref{bmo07}  which is applied separately   for each $\rho_l$, $1\le l\le j$ and for $k+1$. As previously, we fill the annuli if they do not contain any other and define $R_l^*$ in a similar way. The test functions $\vphi$ are build as in \eqref{bmo09}-\eqref{bmo10} with the following new constraint: for $\rho_l$ we take the minimum between each $\vphi_i$ and $ 1- \frac{d(x, \{\rho_l >0\})\wedge R}{R}$,

In this way, the supports of the test functions for $l\not=i$  do not intersect and their gradient is still controlled by $\frac{2}{R_l^*} + \frac 1R$. Then
$$\mu_k(\rho) \le \frac{\big( \frac{2}{R_l^*} + \frac 1R \big)^2(m_l+m_0)}{\frac{c_Nm_l}{k+1}},$$
which leads to inequality \eqref{bmo13}.
\end{proof}
\begin{proof} (of Theorem \ref{bmo06.1}, continuation)
Consider  $(\ov \rho_n)_n$, a maximising sequence for problem \eqref{bmo05},
$$\ov \rho_n=  \rho_1\sqcup \rho_2\sqcup \dots \sqcup \rho_k.$$
In view of the definition of the eigenvalue of $\ov \rho_n$, we can identify the densities $\rho_n^j$ such that
$\int\rho_n^j dx \ra 0$. If this is the case, we just drop them out without decreasing the $k$-th eigenvalue. The new sequence is  also maximizing (possibly after rescalings   in order to satisfy the mass constraint) but with less components.
 For this new maximizing sequence, up to extracting a subsequence and relable the densities composing $\ov \rho_n$, we get that for $1, \dots, j$, with $j \le k$, that
$$\int \rho_n^j dx \ra m_j >0$$
and $(\rho_n^j )_n$ is maximizing sequence for problem \eqref{bmo05}, associated to some eigenvalue index $k_j$ and mass $m_j$.

Relabeling again the indices, it is enough to consider only the case in which a sequence of densities $(\rho_n)_n$ complemented by $0$, i.e. $\ov \rho_n=\rho_n \sqcup  0 \sqcup \dots \sqcup 0$ is maximizing for problem \eqref{bmo05} (possibly with a different $k$ and $m$), hence
$$\mu_k (\rho_n) \ra \sup  \{\mu_k(\rho) : \rho \in\L^1_m(\R^N, [0,1])\} >0.$$

Then $\int_{\R^N} \rho_n dx =m$. We shall use the concentration compactness principle of P.-L. Lions \cite{Li85} to describe the behaviour of the sequence $(\rho_n)_n$. There are three possibilities.
\begin{enumerate}
\item [1.] {\it Compactness.} There exists a subseqence $(\rho_{n_j})_j$ and a sequence of vectors $y_{n_j} \in \R^N$ such that
$$\rho_{n_j} \rau \rho, \mbox{ weakly-* in } L  ^\infty (\R^N)$$
and $\int_{\R^N} \rho dx =m$. Since
$$\mu_k(\rho) \ge \limsup_{j \ra +\infty}  \mu_k(\rho_{n_j}),$$
we conclude with the optimality of $\rho$ from Lemma \ref{bmo15},  since $\rho$ satisfies the mass constraint.
\item [2] {\it Vanishing.} For every $R>0$, we have that
$$\sup_{y \in \R^N} \int_{B(y, R)} \rho_n dx \ra 0, \mbox{ when } \nif.$$
This situation can not occur for a maximizing sequence, as a consequence of the geometric control of the spectrum, Lemma \ref{bmo08}. Indeed, we know that there exists $R^*$ such that
$$\int_{B_{y_n, R^*}} \rho_n \ge \frac{c_N m}{k+1}$$
 uniform in $n$, as soon as $\mu_k(\rho_n) \ge M >0$, which is expected from a maximizing sequence.
\item [3.] {\it Dichotomy.} As vanishing does not occur, we know that there exists a concentration of mass somewhere. Assume that for some subsequence and some translations if necessary we have some concentration of mass
$$m>m_1= \lim_{R\ra +\infty} \lim_\nif  \sup_y \int_{B_{y, R}} \rho_n dx>0.$$
Then for a subsequence, still denoted with the same index, we find a density $\rho_1$ such that
$$\rho_{n} \rau \rho^1, \mbox{ weakly-* in } L ^\infty (\R^N), \int_{\R^N} \rho^1 dx = m_1.$$
We define $R_n$ such that $\int_{B_{R_n} }\rho_n dx \ge m_1-\frac 1n$ and define
$$\rho_n^2= \rho_n\lfloor _{B_{3R_n^c}} \mbox{ and } \rho_n^0 = \rho_n-\rho_n^1-\rho_n^2.$$
Then, $R_n\ra +\infty$,
$$\int_{\R^N}\rho_n^2 dx \ra m-m_1>0  \mbox{ and }  \int_{\R^N}\rho_n^0 dx \ra 0.$$
In view of the enhanced geometric control of the spectrum, Lemma \ref{bmo12}, the sequence of densities $(\rho_n^2)_n$ has also a concentration of mass on balls of uniform radius. Let us denote the maximal mass concentration $m_2>0$.

If $m_2=m-m_1$, this means that the sequence $(\rho_n^2)_n$ satisfies the compactness assumption, so we get a limit $\rho^2$. Then $\ov \rho=\rho_1\cup  \rho_2$ is optimal using the upper semicontinuity result from Lemma \ref{bmo17}   since it satisfies the mass constraint.

If $m_2<m-m_1$, we continue the process  and find another concentration of mass, and so on. This procedure stops after at most $k$ steps, since a density with $k+1$ disjoint concentrations of mass has the $k$-th eigenvalue equal to $0$, so it is not maximizing.

\end{enumerate}

\end{proof}
\begin{remark} {\rm
Let
$F: \R^k_+ \ra \R$
be upper semicontinuous and non decreasing in each variable. Then the following  problem
$$\max \{ F(\mu_1(\ov \rho), \dots, \mu_k(\ov \rho)) : \ov \rho = \rho_1\sqcup \rho_2\sqcup \dots \sqcup \rho_k, \rho_j : \R^N \ra [0,1], \sum _{j=1}^k \int_{\R^N}\rho_j dx =m\}$$
has a solution. A typical example is
$$F(\mu_1(\ov \rho), \dots, \mu_k(\ov \rho))=\sum_{i=1}^k \mu_k(\ov \rho).$$
The proof does not require any further argument with respect to Theorem \ref{bmo06.1}.
}

\end{remark}

\section{The one dimensional case. Proof of Theorem \ref{bmo20.1}}\label{bmo:s3}
This section is devoted to the proof of Theorem \ref{bmo20.1}. The most of results of this section are one dimensional. However, some technical points hold true in any dimension of the space. If this is this case, we shall state the results in the most general framework.

The proof of Theorem \ref{bmo20.1} has two distinctive parts. In the first part, we shall prove the inequality
$$\forall k \in \N, \; \mu_k(\rho) \le \frac{\pi^2k^2}{\Big (\int_\R \rho dx \Big)^2 }$$
for non degenerate densities: $\rho:[0,a]\ra [\delta, 1]$ with $\delta >0$.
In the second part, we shall consider general densities and we shall use an approximation argument. The approximation argument is itself quite technical and requires to regularize the the density in different manners on the numerator and denominator. This is to avoid the presence of spurious modes in the approximation procedure.

We shall use the following classical result on the nodal points of a linear combination of eigenfunctions of a well posed Sturm-Liouville problem, for which we refer to the original paper of Sturm \cite{St36} and to \cite[Theorem 1.4 and Theorem 3.2]{BeHe20}.
\begin{lemma}\label{bmo22}(Extended Courant nodal domain property)
Assume $\rho: [0,a] \ra [\delta, 1]$. If $u_k$ is a nonzero eigenfunction associated to $\mu_k(\rho)$, then $u_k$ has precisely $k$ zeros. Moreover, any linear combination
$$\sum_{i=0}^k a_iu_i$$
of the eigenfunctions $u_0, \dots, u_k$ has at most $k$ zeros, multiplicities being counted.
\end{lemma}
\begin{proof}(of Theorem \ref{bmo20.1}) In a first step, we  assume $\rho: [0,a] \ra [\delta, 1]$ so that the problem is elliptic and the eigenvalue problem is well posed on the interval $(0,a)$. In a second step we shall use a double approximation procedure to deal with degenerate densities, and reduce the argument to the elliptic case.

\smallskip
\noindent{\bf Step 1. Assume  $\rho: [0,a] \ra [\delta, 1]$.}
The proof is based on the construction of a test function $g$ which is $\rho$-orthogonal to the eigenfunctions $1, u_1, \dots, u_{k-1}$ and satisfies
\begin{equation}\label{bmo23}
\frac{\int_0^a\rho (g')^2 dx}{\int_0^a\rho g^2 dx} \le \frac{\pi^2k^2}{\Big (\int_0^a \rho dx \Big)^2 }.
\end{equation}
This will readily give the conclusion. The difficulty to build $g$ is that $u_1, \dots, u_{k-1}$ are not known. This technical point popped up in \cite{BuHe19}, in the case $k=2$ (with only one unknown function) and was dealt with a topological degree argument (see as well \cite{FrLa20}). Here, we have $k-1$ unknown functions but we have the advantage of the dimension being equal to   $1$. We shall also use a topological degree argument, in the spirit of \cite{BuHe19}, however  more sophisticated here, since  orthogonality is searched on $ k$ functions simultaneously.

Assume for now on, without restricting generality, that $\int \rho dx =1$. Let $P= (a_1, \dots, a_k) \in \R^k$ and assume that
$$b_1 \le b_2 \le \dots \le b_k$$
is an ordered relabeling of $a_1, \dots, a_k$. We associate to $P$ the function $g_P$, built as follows. Let first introduce the function (see Figure \ref{bmo24})

$$h(x)= -1_{(-\infty,  -\frac{1}{2k} )} (x)+\sin (k\pi x)1_{[-\frac{1}{2k}, \frac{1}{2k}]} (x)
+1_{( \frac{1}{2k}, +\infty)} (x),$$
\begin{figure}
  \centering
  \includegraphics[width=0.25\textwidth]{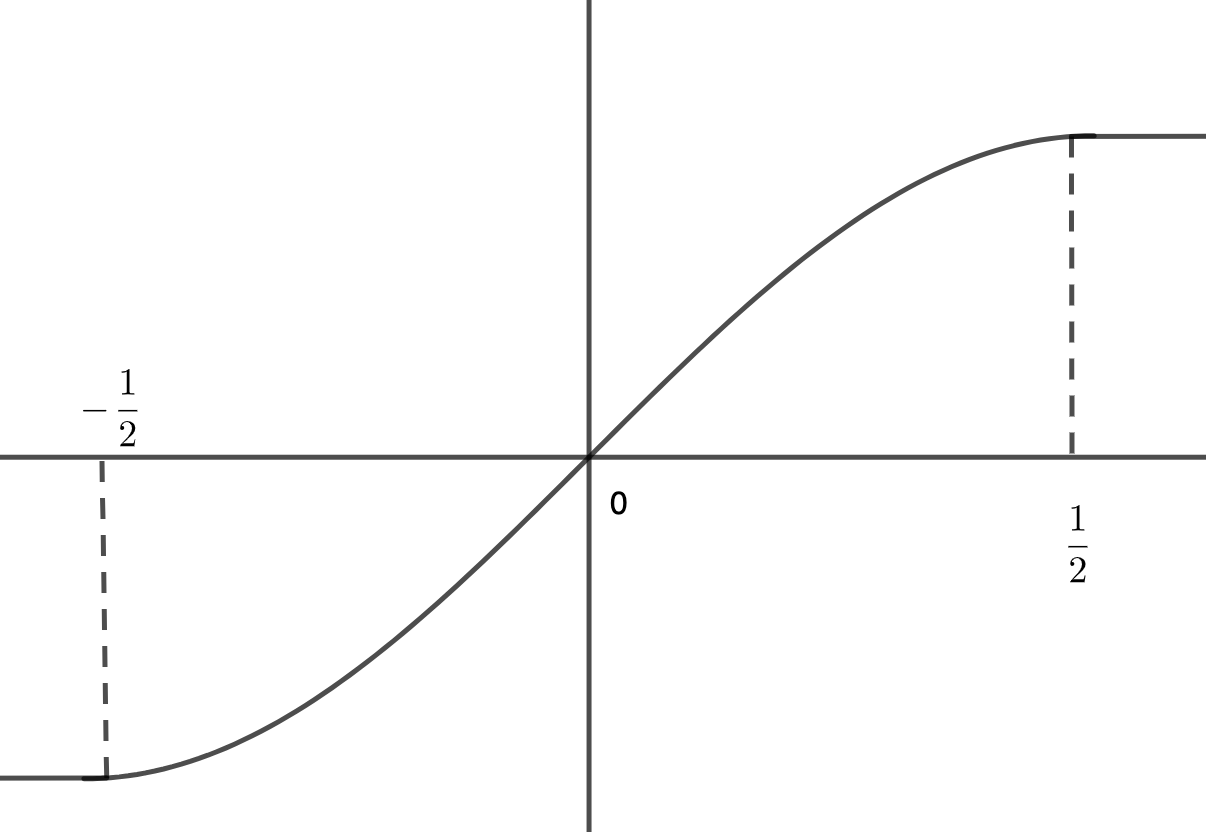}\hskip 2cm
   \includegraphics[width=0.3\textwidth]{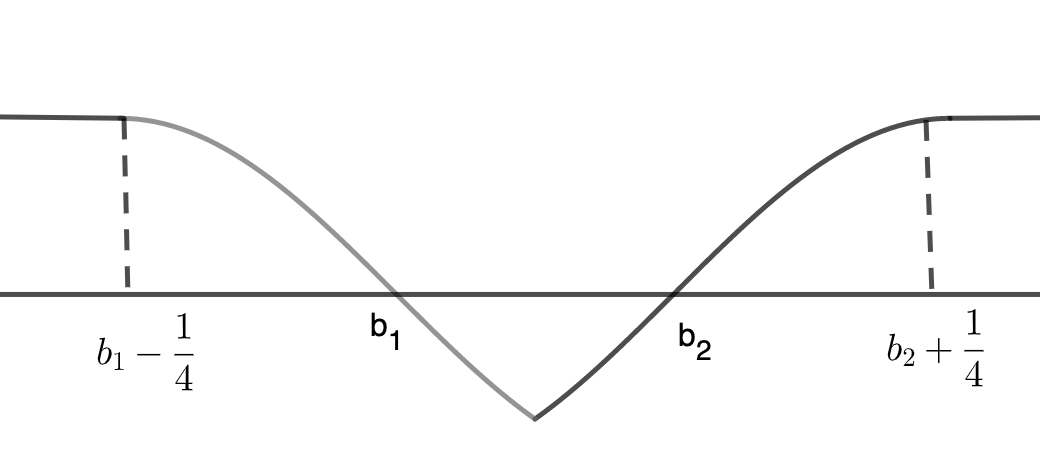}
  \caption{The function $h$ for $k=1$ and a function $g_P$ for $k=2$.}
  \label{bmo24}
\end{figure}

and for every $j =1, \dots, k$ define
$$g_j(x)= h(x-b_j).$$
On any interval $(b_j, b_{j+1})$ we define
$$g_P= (-1)^j\min\{|g_i|, |g_{i+1}|\},$$
while on $(-\infty, b_1)$ $g_P=-g_1$ and on $(b_k, +\infty)$ $g_P= (-1)^kg_{k}$.

\medskip

\noindent{\bf Validity of the functions $g_P$.}
Assume the we find some point $P$ such that
\begin{equation}\label{bmo26}
\forall 0\le i \le k-1, \int_0^a g_Pu_i \rho dx =0.
\end{equation}
Then,
$$\mu_k(\rho) \le \frac{\int_0^a\rho (g_P')^2 dx}{\int_0^a\rho g_P^2 dx}.$$
Then, we use the mass transplantation argument on each interval where $g_P$ equals a function $\pm g_i$ as the maximal interval where the derivative of $g_j$ is not vanishing is of length $\frac 1k$. In this way, we get direct comparison of each ratio
$$\frac{\int_0^a\rho (g_i')^2 dx}{\int_0 ^a\rho g_i^2 dx} \le \frac{\int_{-\frac{1}{2k} }^\frac{1}{2k} \pi^2 k^2 \cos^2(\pi k x) dx}{\int_{-\frac{1}{2k} }^\frac{1}{2k}   \sin^2(\pi k x) dx} = \pi^2k^2,$$
and by standard algebraic argument
$$\frac{\int_0^a\rho (g_P')^2 dx}{\int_0^a\rho g_P^2 dx}\le \pi^2k^2.$$

\medskip

\noindent{\bf Existence of an admissible function $g_P$.}
This is the most technical part. In fact we shall prove a more precise statement in order to be able to use a topological degree argument.
\begin{lemma}\label{bmo25}
There exists a point  $P=(a_1, \dots, a_k)\in \R^k$ satisfying \eqref{bmo26}. Moreover,  this point is unique up to a permutation of the indices and  $-\frac{1}{2k}<a_1 \le\dots\le a_k<a+\frac{1}{2k}$.
\end{lemma}
\begin{proof}
We will use a cyclic induction argument on $k$,
$$\dots \mbox{ existence for } k \Lra \mbox{ uniqueness for } k \Lra   \mbox{ existence for } k+1 \dots $$
\smallskip
\noindent {\bf Start $(k=1)$.} In this case, both existence and uniqueness are trivial. The existence follows as a consequence of the intermediate value property for a continuous function, while uniquness is a consequence of the connectedness of the support of $\rho$. As well, $-\frac 12 < a_1< a+\frac 12$, otherwise $\int_0^a \rho g_Pdx =1\not=0$.

\smallskip
\noindent {\bf Existence for $k$.} Assume the assertion is true up to $k-1$. We shall prove the assertion $k$, starting with the existence part. We know that there exists
$$Q=(b_1, \dots, b_{k-1}) \mbox{ such that. } \forall i=0, \dots , k-2,\;\; \int_\R g_Q u_i \rho dx =0.$$

We fix $M>0$ large enough, say $M \ge a+2k+2$, and define
$$F:[-M, M]^k \ra \R^k,$$
$$F(a_1, \dots, a_k)= (\int _0^a \rho g_P dx,  \int _0^a u_1 \rho g_P dx, \dots, \int _0^a u_{k-1} \rho g_P dx).$$
Above $P=(a_1, \dots, a_k)$. Our purpose is to prove that there exists $P \in [-M, M]^k$ such that $F(P)=0$.

We introduce the segments $S_i=(a+2i, a+2i+\frac 1k)$ and the deformation of $F^t$ of $F$, by
$\forall t \in [0,1]$, $\forall 0\le i\le k-2$, $F_i=F^t_i$ and
$$F^t_{k-1} (P)= t\int \rho u_{k-1} g_P dx + \alpha (1-t) \int _{S_{k}} g_P dx,$$
where $\alpha \in \{-1,1\}$ will be chosen later.

Clearly, this deformation is continuous. Let us prove that no zero crosses $\partial [-M, M]^k$.
Assume for some $t$ that $F(a_1, \dots, a_k)=0$ and that $a_k=M$ or $a_k=-M$. The choice of the index $k$ is possible by a permutation of the indices. Let us assume $a_k=M$. Then, in the definition of $g_P$ the point $a_k$ has no influence on the support of $\rho$. In this way, denoting $P'=(a_1, \dots, a_{k-1})$ we get
$$\forall 0\le i \le k-2, \int \rho u_{i} g_P dx=0.$$
Using the induction hypothesis, the point $P'$ is necessarily the unique point given by the induction hypothesis for $k-1$. Now, we look to the sign of $\int \rho u_{k-1} g_P$. If nonnegative we chose $\alpha=-1$, if nonpositive, we chose $\alpha=1$, so that
$$t\int \rho u_{k-1} g_P dx + \alpha (1-t) \int _{S_{k}} g_P dx\not =0.$$
This choice implies that $F^t$ can not vanish on $[0,1]$. In case $a_k=-M$, the argument is similar.

Once arrived at $t=1$, we relabel $F^1=F$ and continue with a second deformation, keeping constant the last component $\alpha \int _{S_{k}} g_P dx$, keeping unchanged the first $k-2$ components and deforming for $t \in [0,1]$
$$F^t_{k-2} (P)= t\int \rho u_{k-2} g_P dx + \alpha (1-t) \int _{S_{k-1}} g_P dx.$$
This is again continuous and can not have a zero crossing $\partial [-M, M]^k$. The argument is exactly the same. The only new observation comes from the fact that as $ \int _{S_{k}} g_P dx=0$, necessarily one point, say $a_{k}$ has to be the middle of the segement $S_k$. More than one point can not come in touch with $S_{k}$, otherwise strictly less than $k-2$ points would be sufficient to cancel the first $k-2$ components, which is not possible by the induction property.

We continue this procedure and arrive to the final function
$$F_{fin} (P)= (\alpha_1\int_{S_1} g_P dx, \dots, \alpha_k\int_{S_k} g_P dx).$$
This function has as (only) zero any permutation of the centers of the segments $S_i$ and the sign of the Jacobian is always the same, hence is topological degree is not $0$ and we conclude with the existence part of the proof.

\smallskip
\noindent {\bf Uniqueness for $k$.} Assume $P_1\not= P_2$, up to permutations and
$$\forall j=1,2, \forall \; 0\le i \le k-1, \;\;\int \rho g_{P_j} u_i dx =0.$$
We get that $\forall 0 \le i \le k-1, \int \rho (g_{P_1} -g_{P_2})u_i dx =0$. in view of the structures of $g_{P_1}$ and $g_{P_2}$, the function $g_{P_1} -g_{P_2}$ has at most $k-1$ zeroes where there is sign changing.  Assume thay are $y_1, \dots, y_j$, with $j\le k-1$. We find $\alpha_0, \dots, \alpha_j$, not all of them vanishing, such that
$$U=\sum_{i=0}^j \alpha_i u_i$$
satisifies $U(y_i)=0$, $\forall i=1, \dots, j$. This is a consequence of the fact that a linear system with $k$ equations, $k+1$ unknowns and $0$ right hand side, has a space of solutions of co-dimension at least equal to $1$.

Using the extended Courant property,   the number of zeros of $U$ is at most $j$, in our case being exactly equal to $j$. Hence $U$ has to change sign between two zeros, otherwise we would get more than $j$ zeros, including multiplicity. Finally, this implies that
$(g_{P_1}-g_{P_2}) U$
has constant sign. This is a contradiction with the orthogonality property
$$\int \rho (g_{P_1}-g_{P_2}) U dx=0.$$

\smallskip
\noindent {\bf The inequality $-\frac{1}{2k}<a_1 \le\dots\le a_k<a+\frac{1}{2k}$.} In fact, this inequality implies that all $k$-points are necessarily involved in the construction of the function $g$. Assume this inequality is not true. Then any point out of the interval $(-\frac{1}{2k}, a+\frac{1}{2k})$ does not play any role in the geometry of the function $g_P$ on $[0, a]$. Assume
that for some $j <k$, $Q=(a_1, \dots, a_j)$ are the only points in this interval. Then
$$\int g_Q\rho u_i dx =0, i=0, \dots, k-1.$$
The function $g_Q$ changes sign at most $j+1$ times and has no more than $j$ zeros where it changes sign. There exists a linear combination of $u_0, \dots, u_j$ with the same zeros, so that a similar argument as the previous one leads to a contradiction.

\smallskip

\noindent{\bf Step 2.   General densities.} In order to handle the approximation procedure for general densities, we need to regularize differently the numerator and the denominator.  Let $\rho_1, \rho_2 : \R^N\ra [0, +\infty)$, $\rho_1,\rho_2 \in L^1(\R^N) \cap L^\infty (\R^N)$, $\rho_2 \not\equiv0$.
For every integer $k \ge 0$, the number
\begin{equation}\label{bmo03.1}
\mu _k(\rho_1, \rho_2) := \inf_{S\in{\mathcal S}_{k+1}} \max_{u \in S\sm \{0\}} \frac{\int_{\R^N} \rho_1|\nabla u|^2 dx}{\int_{\R^N }\rho_2 u^2 dx},
\end{equation}
is called the $k$-th eigenvalue of the problem
$$-\div(\rho_1\nabla u)= \mu_k (\rho_1, \rho_2) \rho_2 u  $$
with Neumann boundary conditions. Above, ${\mathcal S}_{k+1}$ is the family of all subspaces of dimension $k+1$ in
\begin{equation}\label{bmo03.2}
\{u\cdot 1_{\{\rho_2 (x)>0\}}: u \in C^\infty_c (\R^N)\}.
\end{equation}
If $\rho_1=\rho_2:=\rho$, we have $\mu_k(\rho_1, \rho_2)=\mu_k(\rho)$.
By convention, if $\rho_2=0$, we set $\mu_k(\rho_1, \rho_2)=+\infty$.

The following upper semicontinuity result holds.

\begin{lemma}\label{bmo15}
Assume $\rho_1, \rho_2,\rho_{1}^n \rho_{2}^n  \in L^1(\R^N, [0,1]) $ satisfy $\rho_i^n \rau \rho_i$ weak-* in $L^\infty (\R^N)$, for $i=1,2$. Then
$$\forall k \ge 1, \;\mu_k(\rho_1, \rho_2) \ge \limsup_\nif \mu_k(\rho_1^n, \rho_2^n).$$
\end{lemma}
\begin{proof}
The proof of the result is absolutely similar to   Lemma \ref{bmo15.1}.
\end{proof}

\begin{lemma}\label{bmo15.2}
Let $\delta >0$. Assume $\rho_1, \rho_2 :(0,a) \ra [\delta, 1]$ and denote $m = \min \{ \ds \int_0^a \rho_1 dx, \int_0^a \rho_2 dx\}$. Then,
$$ \mu _k(\rho_1, \rho_2)\le \frac{\pi^2k^2}{ m^2 }.$$
\end{lemma}
\begin{proof}
Assume first that $\int_0^a \rho_2 dx\le \int_0^a \rho_1dx$. The proof    of  Step 1 is valid with  $m =\int_0^a\rho_2dx$. The key point is that the inequality given by mass transplantation   still occurs, since on the regions where $g_P'$ vanishes, the value of $\rho_1$ is not important, while on the regions where  $g_P'$  does not vanish we replace anyway $\rho_1$ by $1$.

In particular, this implies that
\begin{equation}\label{bmo15.7}
 \Big(\int_0^a \rho_2dx \Big )^2\mu_k(1, \rho_2) \le \pi^2k^2.
 \end{equation}
If $\int \rho_2 dx\ge \int \rho_1dx$, for some $0<\alpha <1$ we write $\int \alpha \rho_2 dx= \int \rho_1dx$ and use monotonicity of the Rayleigh quotient together with the previous argument to get
$$\mu_k(\rho_1,   \rho_2)\le \mu_k(\rho_1, \alpha \rho_2) \le \mu_k(1, \alpha \rho_2) \le \frac{\pi^2k^2}{\Big(\ds \int_0^a \alpha \rho_2dx \Big )^2}=\frac{\pi^2k^2}{\Big(\ds \int_0^a   \rho_1dx \Big )^2}.$$
\end{proof}

\noindent{\bf Approximation of density eigenvalues by  regular problems.}
 The following approximation results is inspired from numerical analysis, technically developed in order  to avoid spurious eigenvalues in the numerical approximation (see for instance Allaire and Jouve \cite{allaire_level-set_2005} and Section \ref{bmo:s5}).

  Let  $\rho_1, \rho_2\in L^1(\R^N) \cap L^\infty (\R^N, \R^+)$, $\rho_2 \not\equiv 0$. We denote
 $$\rho_1^\vps= \rho_1+ \vps e^{-\frac{|x|^2}{2}}, \quad \rho_2^\vps= \rho_2 1_{B_{\frac 1\vps}}+ \vps^2 e^{-\frac{|x|^2}{2}}$$
  where $B_r$ stands for the ball of radius $r$ centered at the origin.

The problem associated to $(\rho_1^\vps, \rho_2^\vps) $ is regular. This is a consequence of the compact embedding  $H^1(\R^N, \gamma_N) \sq L^2(\R^N,  \gamma_N)$ and of the ellipticity of the differential operator $ u \to -\div(\rho_1^\vps \nabla u)$ in the associated spaces.

 \begin{lemma}\label{bmo28}
 With the previous notations, we have
 $$\forall k \in \N, \; \mu_k(\rho_1^\vps , \rho_2^\vps) \ra  \mu_k(\rho_1, \rho_2).$$
 \end{lemma}
 \begin{proof}
  The proof is done in two steps: upper and lower semicontinuity. For the upper semicontinuity, it is enough to observe that
    $\rho_i^\vps \rau \rho_i$, weakly-* in $L^\infty(\R^N)$ hence one can relay Lemma \ref{bmo15} to get upper semicontinuity.
Let us prove
\begin{equation}\label{bmo29}
\mu_k (\rho_1, \rho_2)\leq \liminf\limits_{\vps\rightarrow0} \mu_k^\varepsilon (\rho_1^\vps, \rho_2^\vps).
\end{equation}

  To avoid heavy notations, we shall denote $\mu_k^\vps:= \mu_k^\varepsilon (\rho_1^\vps, \rho_2^\vps)$.
The $\vps$-problem is well posed, so that we can consider eigenfunctions $u_0^\vps,...,u_k^\vps \in H^1(\R^N, \gamma_N)$  associated to $\mu_0^\varepsilon,...\mu_k^\varepsilon$ respectively, such that
\begin{equation}\label{bmo30}
  \int_{\R^N}  \rho_2^\vps (u_i^\varepsilon)^2   = 1
\end{equation}
\begin{equation}\label{bmo31.e}
 \int_{\R^N} \rho_1^\vps |\nabla u_i^\varepsilon|^2  = \mu_i^{\varepsilon} \leq M
\end{equation}
\begin{equation}\label{bmo32}
 \forall i \not= j,\;\; \int  \rho_2^\vps u_i^\varepsilon u_j^\varepsilon  = 0
\end{equation}
where $M$ is a common bound for all $\varepsilon$ and all $i\in \{0,...,k\}$ (such a bound exists, otherwise the $liminf$ in \eqref{bmo29} equals $+\infty$ and so there is nothing to prove).
From the first equality we have that $\varepsilon u_i^\varepsilon$ is bounded in $L^2(\R^N, \gamma_N)$. Then up to a subsequence, there exist a $U_i \in L^2(\R^N, \gamma_N)$ such that
\[\varepsilon u_i^\varepsilon \wlim{L^2(\R^N, \gamma_N)} U_i, \;\mbox{weakly}. \]
In particular, if we denote by $\bar{v}:=  \int v d \gamma_N  $, the weighted mean of $v$, we have that
\[\varepsilon \overline{ u_i^\varepsilon} \to \overline{U_i}.\]
This is a consequence of the   fact that $1 \in L^1(\R^N, \gamma_N)$.
From \eqref{bmo31.e} we get that $(\sqrt{\varepsilon}|\nabla u_i^\vps|)$ is bounded in $L^2(\R^N, \gamma_N)$. This implies that $\varepsilon |\nabla u_i^\varepsilon| \xrightarrow{\L^2(\R^N, \gamma_N)} 0$. Thanks to the the compactness of $H^1(\R^N, \gamma_N) $ in $L^2(\R^N, \gamma_N)$, we have that
\[ || \varepsilon u_i^\varepsilon - \varepsilon \overline{u_i^\varepsilon} ||_{L^2(\R^N, \gamma_N)} \to 0 \]
which implies that
\[ \varepsilon u_i^\varepsilon \xrightarrow{L^2(\R^N, \gamma_N)} \overline{U_i}. \]
In the same time,   $u_i^\varepsilon$ is bounded in $\L^2(\R^N, \rho 1_{B_n})$, for some $n \in \N$ such that $|B_n\cap \{\rho >0\}| >0$  so
\[ \varepsilon u_i^\varepsilon \xrightarrow{\L^2(\R^N, \rho 1_{B_n}) } 0 \]
then $\ov U_i = 0$ from the a.e. point wise convergence. In other terms,
\[\varepsilon^2 \int_{\R^N}  (u_i^\vps)^2 e^{-\frac{|x|^2}{2}} dx \to 0\]
then by Cauchy-Schwarz
\[\varepsilon^2 \int_{\R^N}  u_i^\varepsilon u_j^\varepsilon e^{-\frac{|x|^2}{2}}  dx \to 0\]
so that
\[\int_{\R^N}  \rho_2 1_{B_\frac 1\vps}(u_i^\varepsilon)^2 dx \to 1 \]
and
\[\int_{\R^N}  \rho_2 1_{B_\frac 1\vps}u_i^\vps u_j^\vps  dx\to 0\]

Finally,
\[\mu_k^\varepsilon = \frac{\ds \int_{\R^N}  \rho_1 |\nabla u_k^\varepsilon|^2 + \varepsilon e^{-\frac{|x|^2}{2}}  |\nabla u_k^\varepsilon|^2}{\ds \int_{\R^N} \rho_21_{B_\frac 1\vps} (u_k^\varepsilon)^2 + \varepsilon^2  e^{-\frac{|x|^2}{2}} (u_k^\varepsilon)^2}\]
\[
= \max\limits_{v^\varepsilon \in Span(u_0^\varepsilon,...,u_k^\varepsilon)} \frac{\ds \int_{\R^N} \rho_1 |\nabla v^\varepsilon|^2 + \varepsilon e^{-\frac{|x|^2}{2}}  |\nabla v^\varepsilon|^2}{\ds \int_{\R^N} \rho_2 1_{B_\frac 1\vps}(v^\varepsilon)^2 + \varepsilon^2e^{-\frac{|x|^2}{2}}  (v^\varepsilon)^2}
\geq \max\limits_{v^\vps \in Span(u_0^\varepsilon,...,u_k^\varepsilon)} \frac{\ds \int_{\R^N}  \rho_1 |\nabla v^\varepsilon|^2 }{\ds \int_{\R^N}  \rho_2 (v^\varepsilon)^2 + \varepsilon^2 e^{-\frac{|x|^2}{2}} (v^\varepsilon)^2}.\]
Passing to the $\liminf$ and recalling that
\[\ds \int_{\R^N}  \varepsilon^2 e^{-\frac{|x|^2}{2}} (v^\varepsilon)^2 \to 0 \]
for every $v^\varepsilon= \sum _{j=0}^k a_j^\vps u_j^\vps$ with $\sum _{j=0}^k (a_j^\vps )^2 =1$,
we get
\[\liminf \mu_k^\varepsilon
\geq \liminf \max\limits_{v^\varepsilon \in Span(u_0^\varepsilon,...,u_k^\varepsilon)} \frac{\ds \int_{\R^N} \rho_1 |\nabla v^\varepsilon|^2 }{\ds \int_{\R^N}  \rho_2 (v^\vps)^2} \]
We conclude by remarking that $Span(u_0^\varepsilon,...,u_k^\varepsilon)1_{\{\rho_2>0\}}$ is of dimension $k+1$ for $\varepsilon$ small enough,
hence
\[\max\limits_{v^\varepsilon \in Span(u_0^\varepsilon,...,u_k^\varepsilon)} \frac{\ds \int_{\R^N}  \rho_1 |\nabla v^\varepsilon|^2 }{\ds \int_{\R^N}  \rho_2 (v^\varepsilon)^2}
\geq \mu_k(\rho_1, \rho_2).\]
 \end{proof}

\smallskip
\noindent {\it Proof of Theorem \ref{bmo20.1}, continuation.}
 Let now $\rho \in \L^1(\R, [0,1])$. For every $\vps >0$, we define
$$\mu_k^\vps(\rho) = \inf_{U \in S_{k+1}} \max _{u \in U} \frac{\ds \int_{\R} (\rho + \vps e^{-x^2})(u')^2 dx }{\ds \int_{\R} (\rho 1_{[-\frac 1\vps, \frac 1\vps]}+ \vps^2 e^{-x^2})u^2 dx }.$$
This problem is well posted in view of the compact embedding of the Sobolev space associated to the Gaussian measure in $L^2$. Moreover $\mu_k^\vps (\rho) \ra \mu_k(\rho)$ when $\vps \ra 0$.

Let is introduce the second approximation. For every $n \in \N$,
$$\mu_k^{\vps, n} (\rho) = \inf_{U \in S_{k+1}} \max _{u \in U} \frac{\ds \int_{-n} ^n(\rho + \vps e^{-x^2})(u')^2 dx }{\ds \int_{-n} ^n(\rho 1_{[-\frac 1\vps, \frac 1\vps]}+ \vps^2 e^{-x^2})u^2 dx }.$$
Then,   for every fixed $\vps >0$, we have
$$ \forall k \ge 1, \; \mu_k^{\vps, n} (\rho)  \ra \mu_k^\vps(\rho), \mbox{ when } \nif.$$
Indeed, the upper semicontinuity is done as before. The lower semicontinuity is a consequence of the "collective compactness" property: if $\vphi_n \in H^1(-n,n)$ is such that $\int_{-n}^n (\vphi_n')^2 e^{-x^2} dx \le M$ and
$\int_{-n}^n \vphi_n^2 e^{-x^2} dx=1$, then there exists $\vphi \in H^1(\R, \gamma_1)$ such that
$\vphi_n \rau \vphi$ in $H^1_{loc} (\R)$ and
$$\vphi_n 1_{(-n,n)} \ra \vphi, \mbox{ strong in } L^2(\R, \gamma_1).$$
This comes from  the compact embedding $H^1(\R, \gamma_1) \sq L^2(\R, \gamma_1) $ and the uniform extension operator
$$H^1((-n,n), \gamma_1) \ni \vphi \ra \tilde \vphi \in H^1(\R, \gamma_1),$$
given by the reflection agains the points $n+2n\Z$.

Denoting
$$m_{\vps, n}:= \min \Big\{ \frac{\ds \int_{-n} ^n\rho + \vps e^{-x^2} dx}{1+\vps}, \frac{\ds \int_{-n} ^n\rho 1_{[-\frac 1\vps, \frac 1\vps]}+ \vps^2 e^{-x^2} dx}{1+\vps^2}\Big \},$$
we get, from Lemma \ref{bmo15.2}
$$\mu_k^{\vps, n} (\rho) \le \frac{1+\vps}{1+\vps^2} \cdot \frac{\pi^2 k^2}{m_{\vps, n}^2}.$$
Passing to the limit $\nif$ and then $\vps \ra 0$, we conclude the proof.

\end{proof}

\end{proof}

\medskip

\begin{remark}[The equality case]  Clearly, if $\rho : [0,a] \ra[\delta, 1]$ satisfies
$$\mu_k(\rho)= \frac{\pi^2 k^2}{\Big (\ds \int_0^a \rho dx \Big) ^2},$$
  this means that the function $g_P'$ has to be an eigenfunction and all the mass of $\rho$ is distributed on the $k$ segments of lenght $\frac 1k \ds \int_0^a \rho dx$, which have to be disjoint. This is a consequence of the mass transfer procedure. However, the hypothesis $\rho \ge \delta 1_{[0,a]}$ prevents this situation.
On the other hand, if $\rho$ equals the characteristic function of $k$ such segments, then equality occurs.

More interesting is the anlaysis of the equality case in Lemma \ref{bmo15.2}. Assume now  $\rho$ equals the characteristic function of $k$
$$\mu_k(1,\rho)= \frac{\pi^2 k^2}{\Big (\ds \int_0^a  \rho dx \Big) ^2}.$$
Then
$$ \frac{\pi^2 k^2}{\Big (\ds \int_0^a \rho dx \Big) ^2}= \mu_k(\rho) \le \mu_k(1,\rho) \le   \frac{\pi^2 k^2}{\Big (\ds \int_0^a  \rho dx \Big) ^2}.$$
In fact, in general, the eigenvalues $\mu_j(1, \rho)$ depend on the pairwise distance between the segments and
$$\mu_j(\rho) \le \mu_j(1,\rho).$$
The inequality is, in general, strict, except for the $k$-th eigenvalue, when the values are necessarily equal.
\end{remark}
We end this section with the following sharp inequality involving the Sturm-Liouvielle eigenvalues.
\begin{corollary}[Sharp inequalities for Sturm-Liouville eigenvalues] Let $(\alpha,\beta)\sq \R$ be an interval and $\rho_1,\rho_2:[\alpha, \beta]\ra \R$ be positive  $C^1$ functions. We consider the eigenvalue problem
\begin{equation*}
    \left\{
      \begin{aligned}
         -(\rho_1u')'  = \mu_k\rho_2 u  \text{ on } (\alpha, \beta)
        \\ u'(\alpha)=u'(\beta)  = 0
      \end{aligned}
    \right.
\end{equation*}
Then
\[\forall k \ge 0, \; \mu_k\le \frac{\|\rho_2\|_\infty}{\|\rho_1\|_\infty}\frac{\pi^2k^2}{ \min (\frac{\int_\alpha ^\beta \rho_1}{\|\rho_1\|_\infty}, \frac{\int_\alpha ^\beta \rho_2}{\|\rho_2\|_\infty})^2 }
\]

\end{corollary}
\begin{proof}
This is a consequence of Theorem \ref{bmo20.1} applied to $\frac{ \int_\alpha ^\beta  \rho_1}{\|\rho_1\|_\infty}$, $\frac{\int_\alpha ^\beta  \rho_2}{\|\rho_2\|_\infty}$ and relies on the argument of Lemma \ref{bmo15.2}.
\end{proof}

\section{P\`olya conjecture and Kr\" oger inequalities}\label{bmo:s4}

\medskip
A direct consequence of Theorem \ref{bmo20.1} is the following.
\begin{corollary}
The P\'olya conjecture holds in one dimension of the space. Let $\rho \in L^1(\R, [0,1])$, $\rho \not \equiv 0$. Then
\begin{equation}\label{bmo21}
\forall k \in \N, \; \mu_k(\rho) \le \frac{\pi^2k^2}{\Big (\int_\R \rho dx \Big)^2 }.
\end{equation}
\end{corollary}
Although the conjecture is not proved in general in any dimension of the space, a natural question is whether the validity of the conjecture could cover the density eigenvalues. This is the case for $k=1,2$. To bring   support in this direction, our purpose is to prove that estimates obtained by Kr\"oger  for smooth sets (see  \cite{Kr92} ) are still valid in our context of  (degenerate) densities.  As we shall see below, the estimates hold true and naturally involve the $L^\infty$ and $L^1$ norms of the densities. The proof is, in its main lines, the same, except the need of approximation of degenerate densities by regular problems.

\begin{proof}[Proof of Theorem \ref{bmo33.1}]
Let $\rho \in \L^1(\R^N) \cap \L^\infty(\R^N, \R^+)$, $\rho \not \equiv 0$. We want to prove that
\[\forall k \in \N, \;\;\; \mu_k(\rho) \leq 4\pi^2\Bigg( \frac{(N+2)k}{2\omega_{N}}\frac{||\rho||_\infty}{||\rho||_1} \Bigg)^{2/N}.\]

From Lemma \ref{bmo28} we know that
$$\mu_k(\rho + \vps e^{-|\cdot|^2}, \rho \1_{\B_{1/\vps}} + \vps^2 e^{-|\cdot|^2}) \ra \mu_k(\rho) \mbox{ for } \vps \ra 0.$$
Consequently, it is enough to prove  inequality \eqref{bmo377} below, for the couple of densities $(\rho + \vps e^{-|\cdot|^2},  \rho \1_{\B_{1/\vps}} + \vps^2 e^{-|\cdot|^2})$ and then to pass to the limit
\begin{equation}\label{bmo377}
\mu_k (\rho_1, \rho_2) \leq 4 \pi^2 \frac{||\rho_1||_1}{||\rho_2||_1}\Bigg(\frac{ (N+2)k}{2\omega_{N}} \frac{||\rho_2||_\infty}{||\rho_2||_1}\Bigg)^{2/N}.
\end{equation}

So, without restricting the generality, we can assume that
$\rho_1,\rho_2\in \L^1(\R^N)\cap\L^\infty(\R^N)$ are such that there exists $\alpha, \beta > 0$ verifying
\[ \alpha e^{-\frac{|x|^2}{2}} \leq \rho_1, \rho_2  \mbox{ and } \rho_2 \leq \beta e^{-\frac{|x|^2}{2}} .\]
In our case, $\alpha= \vps^2$ and $\beta= (1+\vps^2)e^{\frac{1}{2\vps^2}}$.
Then the eigenvalue problem

\begin{equation*}
    \left\{
      \begin{aligned}
         -\div[\rho_1\nabla u] = & \mu_k (\rho_1, \rho_2) \rho_2 u  \text{ on } \R^N
        \\ u \in H^1(\R^N, \gamma_N)
      \end{aligned}
    \right.
\end{equation*}

is well posed.

 We   detail the proof  of inequality \eqref{bmo377} in order to show how to handle the presence of densities, but refer to the paper of Kr\"oger  \cite{Kr92}  for the original proof (see also \cite{CES15,CoEl19}).

 Let $u_0,...,u_{k-1}$ be eigenfunctions associated to $\mu_0,...,\mu_{k-1}$ which satisfy  $\int_{\R^n} \rho_2 u_i u_j = \delta_{ij}$.
Let
\[\Phi(x,y) = \sum\limits_{j=0}^{k-1} \rho_2(x)u_j(x)u_j(y)\;
\mbox{ and } \; h_z(y)  =e^{i z\cdot y}.\]
If
\[\hat\Phi(z,y) = \frac{1}{(2\pi)^\frac2N}\int_{\R^N} e^{iz\cdot x}\Phi(x,y) \d x\]
denotes the Fourier transform of $\Phi(.,y)$ then for every $z \in \R^N$ the function $(2\pi)^\frac2N\hat\Phi(z,.)$ is the orthogonal projection of $h_z$ on $Span (u_0,...,u_{k-1})$ for the scalar product $(u,v) \to \int_{\R^n} \rho_2 uv$.
Hence
\[\mu_k (\rho_1, \rho_2) \leq  \frac{\int_{\R^N} \rho_1|\nabla(h_z(y) - (2\pi)^\frac2N\hat\Phi(z,y))|^2 \d y}
{\int_{\R^N}\rho_2|(h_z(y) - (2\pi)^\frac2N\hat\Phi(z,y)|^2\d y},\]
then by summing with respect to $z$ over $\B_r$ for some $r>0$, we get
\[\mu_k(\rho_1, \rho_2) \leq  \frac{\int_{\B_r}\int_{\R^N} \rho_1|\nabla(h_z(y) - (2\pi)^\frac2N\hat\Phi(z,y))|^2 \d y \d z}
{\int_{\B_r}\int_{\R^N}\rho_2|(h_z(y) - (2\pi)^\frac2N\hat\Phi(z,y)|^2 \d y \d z}.\]

Let's compute first the numerator:
$$
\int_{\B_r}\int_{\R^N} \rho_1|\nabla(h_z(y) - (2\pi)^\frac2N\hat\Phi(z,y))|^2 \d y \d z = \int_{\B_r}\int_{\R^N} \rho_1 |\nabla h_z(y)|^2 \ y \d z$$
    $$- 2 Re \int_{\B_r}\int_{\R^N} \rho_1 \nabla(h_z(y) - (2\pi)^\frac2N\hat\Phi(z,y))\cdot \nabla((2\pi)^\frac2N\hat\Phi(z,y)) \d y \d z  $$
$$    - (2\pi)^n \int_{\B_r}\int_{\R^N} \rho_1 |\nabla \hat\Phi(z,y)|^2 \d y \d z.
$$
Since $|\nabla h_z(y)|=|z|$, we get
\[ \int_{\B_r}\int_{\R^N} \rho_1 |z|^2 \d y \d z = ||\rho_1||_1 \frac{r^{N+2}}{N+2}N \omega_{N} \]
We have
$$
Re \sum\limits_{j=0}^k \int_{\B_r}\int_{\R^N} \rho_1 \nabla( h_z(y) - (2\pi)^\frac2N\hat\Phi(z,y))\cdot \nabla(\widehat{\rho_2 u_j}(z)u_j(y) \d y \d z$$
$$=Re \sum\limits_{j=0}^k \int_{\B_r} \widehat{\rho_2 u_j}(z) \int_{\R^N} \rho_1 \nabla( h_z(y) - (2\pi)^\frac2N\hat\Phi(z,y))\cdot \nabla u_j(y) \d y \d z=0.
$$
This is a consequence of the orthogonality hypothesis together with the fact that  $u_j$ is an eigenfunction. Indeed,
\[ \int_{\R^N} \rho_1 \nabla( h_z(y) - (2\pi)^\frac2N\hat\Phi(z,y))\cdot \nabla u_j(y) \d y = \mu_j \int_{\R^N} \rho_2 ( h_z(y) - (2\pi)^\frac2N\hat\Phi(z,y)) u_j(y) \d y = 0.\]
We conclude that an upper bound of the numerator is $||\rho_1||_1 \frac{r^{N+2}}{N+2}N\omega_{N}$.

Concerning the denominator, we similarly have

$$ \int_{\B_r}\int_{\R^N}\rho_2|(h_z(y) - (2\pi)^\frac2N\hat\Phi(z,y)|^2 \d y \d z
    = \int_{\B_r}\int_{\R^N} \rho_2 |h_z(y)|^2 \d y \d z$$
    $$ - 2 Re \int_{\B_r}\int_{\R^N} \rho_2 (h_z(y) - (2\pi)^\frac2N\hat\Phi(z,y))((2\pi)^\frac2N\hat\Phi(z,y)) \d y \d z  $$
   $$ - (2\pi)^N \int_{\B_r}\int_{\R^N} \rho_2 |\hat \Phi(z,y)|^2 \d y \d z$$

We get
\[\int_{\B_r}\int_{\R^N} \rho_2 |h_z(y)|^2 \d y \d z = ||\rho_2||_1 r^N \omega_{N},\]
$$  Re \int_{\B_r}\int_{\R^N} \rho_2 (h_z(y) - (2\pi)^\frac2N\hat\Phi(z,y))((2\pi)^\frac2N\hat\Phi(z,y)) \d y \d z =0,$$
\[(2\pi)^N \int_{\B_r}\int_{\R^N} \rho_2 |\hat \Phi(z,y)|^2 \d y \d z= (2\pi)^N \sum\limits_{j=0}^k \int_{\B_r} |\widehat{\rho_2 u_j}|^2 \d z \]
Now $\int_{\B_r} |\widehat{\rho_2 u_j}|^2 \d z \leq \int_{\R^n} |\widehat{\rho_2 u_j}|^2 \d z$. This last term is finite since $\rho_2 u_j^2 \in \L^1(\R^N)$  and $\rho_2$ is bounded, so that $\rho_2 u_j \in L^2(\R^N)$. Using Plancherel's theorem

\begin{align*}
    \int_{\R^N} |\widehat{\rho_2 u_j}|^2 \d z & = \int_{\R^N} |\rho_2 u_j|^2 \d z\\
    & \leq ||\rho_2||_\infty \int_{\R^N} \rho_2 |u_j|^2 \d z\\
    & = ||\rho_2||_\infty
\end{align*}

Now if $\frac{r^N}{N} ||\rho_2||_1 N\omega_{N} - (2\pi)^N k ||\rho_2||_\infty > 0$, we have :

\[\mu_k \leq \frac{\frac{r^{N+2}}{N+2} ||\rho_1||_1 N\omega_{N}}{\frac{r^N}{N} ||\rho_2||_1 N\omega_{N} - (2\pi)^N k ||\rho_2||_\infty} \]
which leads to the desired result when
\[r = 2\pi\Bigg(\frac{(N+2)k ||\rho_2||_\infty}{2 \omega_{N} ||\rho_2||_1} \Bigg)^{1/N}.\]

\end{proof}

\section{Numerical approximation of optimal densities}\label{bmo:s5}
In this section we describe a numerical approximation process of optimal densities  in $\R^2$ for $k=1, \dots, 8$. For $k=1,2$ we recover the theoretical results proved in \cite{BuHe19}, namely the ball and the union of two equal balls respectively. It is a challenge to prove or disprove, for $k \ge 3$, that the optimal densities correspond to a characteristic function. The numerical simulations suggest, to some extent, that the optimal densities are not  characteristic functions. We compare the optimal densities with previous computations of optimal shapes and we obtain reliable better values. However, the excess of the optimal eigenvalues in the class of densities over the optimal eigenvalue in the class of domains is relatively low.

\bigskip
\noindent{\bf Spurious modes.} Let $D=(-1,1)^N\sq \R^N$.   For $\vps>0$ we denote by $\mu_k^\vps(\rho)$ the $k$-th eigenvalue of the well posed elliptic problem
\begin{equation}
    \left\{
      \begin{aligned}
         -\div[(\rho+\vps)\nabla u] = & \mu_k^\vps(\rho) (\rho+\vps^2) u  \text{ on } D,
        \\ \partial_n u & = 0 \text{ on } \partial D.
      \end{aligned}
    \right.
    \label{eq:approx_pb}
\end{equation}
The following occurs.
\begin{lemma}\label{bmo28.7}
  Let $D$ be a bounded, open Lipschitz set and $\rho\in  L^\infty (D, \R^+)$, $\rho\not\equiv 0$. With the notations above, for $\vps \ra 0$ we have
 $$\forall k \in \N, \; \mu_k^\vps (\rho) \ra  \mu_k(\rho).$$
 \end{lemma}
 \begin{proof}
 The proof is similar to Lemma \ref{bmo28}, so we do not reproduce it here.
 \end{proof}
We rely on this approximation result in order to numerically estimate $\mu_k(\rho)$.
Notice the difference of homogeneity of the terms depending on $\vps$ and $\vps^2$.  This kind of formulation has been introduced by Allaire and Jouve in \cite{allaire_level-set_2005} in the context of the numerical optimization of mechanical structures and is a crucial point to avoid spurious modes.

Indeed,  assume $\Om \sq D$ is a smooth domain and consider the well posed elliptic problem
\begin{equation} \label{bmoo1}
\begin{cases}
  -\div[(\1_\Omega+\vps)\nabla u]  = \mu (\1_\Omega+\vps) u  \text{ on } D,
 \\ \partial_n u = 0 \text{ on } \partial D
\end{cases}
\end{equation}
expecting that the spectrum of this problem would converge to Neumann spectrum on $\Om$.
 In this natural, but naive, relaxation procedure, extra spurious eigenvalues are actually expected. These eigenmodes are solutions of the limit eigenvalue problem:
$$
\begin{cases}
-\Delta u = \mu u  \text{ on } D \setminus \Omega,
\\ u = 0 \text{ on } \Omega,
\\ \partial_n u = 0 \text{ on } \partial D.
\end{cases}
$$
We observe numerically that the associated extra eigenfunctions have a Dirichlet energy concentrated in the complement of the domain $\Om$! As an example, we plot the graph of the first spurious eigenfunction for $D = (-1,1)^2$ and $\Omega$ the centered ball of radius $r=0.4$ in Figure ~\ref{fig:spurious}. Values in Table ~\ref{fig:spurious_persistence} illustrate the persistence of spurious modes even for small values of $\vps$.

\begin{table}
  \begin{tabular}{|l||l|l|l|l|l|l|l||l|l|l|}
  \hline
   & \multicolumn{7}{|c||}{$\vps - \vps$} & \multicolumn{3}{|c|}{$\vps - \vps^2$}\\\hline
  $\vps$ & $\mu_1$ & $\mu_2$ & $\mu_3$ & $\mu_4$ & $\mu_5$ & $\mu_6$ & $\mu_7$ & $\mu_1$ & $\mu_2$ & $\mu_3$ \\ \hline
  0.1 & 18.50 & 18.50 & 37.36 & 46.16 & 49.84 & 49.84 & 54.46 & 23.24 & 23.24 & 64.29  \\ \hline
 0.05 & 19.64 & 19.64 & 40.91 & 47.10 & 49.28 & 49.28 & 56.08 & 22.34 & 22.34 & 61.80  \\ \hline
 0.01 & 20.81 & 20.81 & 46.07 & 48.06 & 48.63 & 48.63 & 57.64 & 21.41 & 21.41 & 58.92  \\ \hline
 0.005 & 20.98 & 20.98 & 47.14 & 48.20 & 48.53 & 48.53 & 57.86 & 21.28 & 21.28 & 58.51  \\ \hline
 0.001 & 21.11 & 21.11 & 48.20 & 48.32 & 48.45 & 48.45 & 58.04 & 21.18 & 21.18 & 58.17  \\ \hline
 0.0005 & 21.13 & 21.13 & 48.33 & 48.36 & 48.44 & 48.44 & 58.06 & 21.16 & 21.16 & 58.12  \\ \hline
 0.0001 & 21.15 & 21.15 & 48.35 & 48.43 & 48.43 & 48.48 & 58.08 & 21.15 & 21.15 & 58.09  \\ \hline
 5e-05 & 21.15 & 21.15 & 48.35 & 48.43 & 48.43 & 48.50 & 58.08 & 21.15 & 21.15 & 58.09  \\ \hline
 1e-05 & 21.15 & 21.15 & 48.35 & 48.43 & 48.43 & 48.51 & 58.08 & 21.15 & 21.15 & 58.08  \\ \hline

  \end{tabular}
  \smallskip
\caption{Spurious modes persistence. Eigenvalues computed with the approximation \eqref{bmoo1} containing spurious modes (left) and with the approximation  \eqref{eq:approx_pb} (right)}
 \label{fig:spurious_persistence}
\end{table}

\begin{figure}
  \centering
  \includegraphics[width=0.5\textwidth]{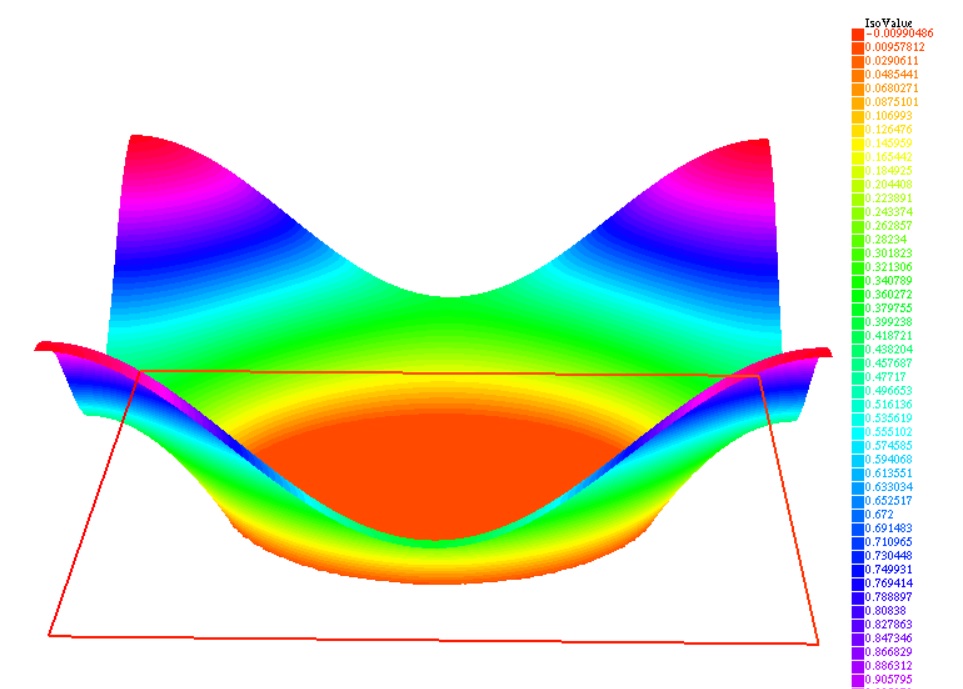}
  \caption{One spurious mode}
  \label{fig:spurious}
\end{figure}

Coming back to problem \eqref{bmo01}, our numerical approximation  is fully justified by the avoidance of spurious modes from  Lemma  \ref{bmo28.7} and by the following proposition.
\begin{proposition}\label{bmo31}
 Under the  notations of  \eqref{eq:approx_pb}, we have
\[ \max\limits_{\rho\in\L_m^1(D,[0,1])} \mu_k^\vps(\rho) \xrightarrow[\vps \to 0]{} \max\limits_{\rho\in\L_m^1(D,[0,1])} \mu_k(\rho).\]
\end{proposition}

\proof  Let $\rho^* \in\L_m^1(D,[0,1])$ be a maximizer of $\mu_k$ and $\rho_\vps\in\L_m^\infty(D,[0,1])$ be a maximizer of $\mu_k^\vps$ (their existence are immediate in view of Lemma \ref{bmo15}).
Without restricting generality, we can assume that
$$\rho_\vps \xrightharpoonup{\L^\infty-*} \Tilde\rho$$
and $\Tilde\rho \in \L^1_m(D, [0,1])$.

Then,      on  the one hand
$$\limsup\limits_{\vps \ra 0} \mu_k^\vps(\rho_\vps) \leq \mu_k(\Tilde\rho)\le \mu_k(\rho^*). $$
On the other hand,
\[ \mu_k^\vps(\rho^*) \leq \mu_k^\vps(\rho_\vps) \]
so that passing to the limit we get
$$\mu_k(\rho^*) \le \liminf_{\vps \ra 0}  \mu_k^\vps(\rho_\vps),$$
which implies that $\tilde \rho$ is also maximizer for $\mu_k$ and concludes the proof.
\qed

\bigskip
\noindent{\bf Numerical strategy.}
  We fix $\vps, m > 0$  and  look for  a numerical solution of
$$\max\{\mu_k^\vps(\rho) : \rho \in \L^1_m(D, [0,1])\}.$$
We implemented a finite element method for the computation of eigenvalues.
Assume $D $ is meshed by a union of triangles $(T_p)_p$. Densities from the  space $\L^1_m(D, [0,1])$ are approximated by standard P1 elements consisting of continuous piecewise linear functions on the given mesh. We denote this space  by $V_h$. Functions of $H^1(D)$ are approximated by the space $U_h$ of standard P2 elements consisting of continuous piecewise quadratic polynomials. The canonical basis of $V_h$ (resp. $U_h$)  is denoted by $(\phi_i)_i$ (resp. $(\psi_i)_i$).
For a given $\rho \in V_h$ we denote by $\rho_i$ it's i-th coordinate in the canonical basis of $V_h$.
The discretized version of our eigenvalue problem is then:
\[
 \vec{M}^\vps(\rho)\vec{u}_k^\vps(\rho) = \mu_k^\vps(\rho)  \vec{K}^\vps(\rho)\vec{u}_k^\vps(\rho)
\]
where
\[\vec{M}^\vps(\rho)_{i,j} = \int_D (\rho+\vps) \nabla \psi_i \nabla \psi_j, \]
\[\vec{K}^\vps(\rho)_{i,j} = \int_D (\rho+\vps^2) \psi_i \psi_j\]
and $\vec{u}_k^\vps(\rho)$ are the coordinates of $u_k^\vps(\rho)$ in $(\psi_i)_i$.

In order to implement a gradient descent algorithm, we compute the discrete derivative of our functional
\begin{align*}
  \mu_k^\vps \colon V_h &\to \R\\
  \rho &\mapsto \mu_k^\vps(\rho).
\end{align*}
For a fixed $\vps>0$ we are interested in solving the following maximization  problem:
\begin{maxi*}|s|
  {\rho \in V_h}{\mu_k^\vps(\rho).}
  {}{}
  \addConstraint{0 \leq \rho \leq 1}
  \addConstraint{\int_D\rho=m.}{}
\end{maxi*}
Using homogeneity arguments, we can get rid of the equality constraint by considering the unconstrained problem
\begin{maxi}|s|
  {\rho \in V_h}{ \| \rho \|_1 \mu_k^\vps(\rho) - \alpha ( \| \rho \|_1- m)^2}
  {}{}
  \addConstraint{0 \leq \rho \leq 1}
\label{eqn:nonsmooth}
\end{maxi}
where $\alpha > 0$ is a parameter related to the mass constraint.
When the eigenvalue is simple, a straightforward direct computation of the $\rho_l$-th partial derivative of $\mu_k^\vps$ gives
\[
\partial_{\rho_l} \mu_k^\vps = \frac{(\vec{u}_k^\vps)^\top \partial_{\rho_l} \vec{M}^\vps \vec{u}_k^\vps
-\mu_k^\vps (\vec{u}_k^\vps)^\top \partial_{\rho_l} \vec{K}^\vps \vec{u}_k^\vps}
{(\vec{u}_k^\vps)^\top \vec{K}^\vps \vec{u}_k^\vps}
\]
where
\[\partial_{\rho_l} \vec{M}^\vps(\rho)_{i,j} = \int_{D} \phi_{l} \nabla \psi_i \nabla \psi_j \]
and
\[\partial_{\rho_l} \vec{K}^\vps(\rho)_{i,j} = \int_{D} \phi_{l} \psi_i \psi_j. \]

It is standard in eigenvalue optimization to expect a multiplicity higher than one for optimal parameters. This phenomenon is known to produce numerical difficulties in the approximation procedure related to the non differentiability of the functional \cite{MR1344670} \cite{MR1624599}. A  multiplicity higher than one causes the gradient direction to greatly change from one iteration to the next one close to the optimum causing instabilities preventing the algorithm to converge.
On the other hand, symmetric combination of multiple eigenvalues are known to preserve smoothness \cite{kato2013perturbation}.

Having this observation in mind, we introduce a straightforward (but efficient enough in our context) two phases process: When the optimization of problem (\ref{eqn:nonsmooth}) encounters a multiplicity higher than one preventing the gradient based approach to converge we switch to the equivalent smoother problem defined by
\begin{maxi*}|s|
  {\rho \in V_h}{ \| \rho \|_1 \sum_{i=0}^{l-1} \mu_{k+i}^\vps(\rho) - \alpha ( \| \rho \|_1- m)^2}{}{}
  \addConstraint{0 \leq \rho \leq 1}
  \addConstraint{\sum_{0 \leq i<j \leq l-1} (\mu_{k+i}(\rho) - \mu_{k+j}(\rho))^2=0}
\end{maxi*}
where $l = \max\{ l\geq 0 : \mu_{k+l}^\vps(\rho)- \mu_{k}^\vps(\rho) < \sigma\}$ for some small parameter $\sigma>0$. Notice that in this new, trivially equivalent, formulation, cost and constraint functions are both smooth in the neighborhood of the optimal density.

\bigskip
\noindent{\bf Technical details.}
The optimization process has been carried out by IPOPT, an interior point optimization method \cite{wachter2006implementation}. Each side of the domain $D=[0,1]^2$ is discretized into $100$ segments then triangulated leading to a total of $11658$ degrees of freedom for the function space $V_h$. The finite elements computations have been carried out through GetFEM \cite{MR4199501}.
We used the values $\vps = 0.001$, $m=0.4$ and $\sigma=0.1$. 

After the optimization procedure, we perform a final estimation of the eigenvalue in a post-processing phase. This post-processing consists in getting rid of the areas where the density is almost zero and recomputing the eigenvalue on the remaining density on the smaller domain without using the relaxation. To this purpose, we used MMG \cite{MR3163123} to mesh the domain $\{\rho > 0.01 \}$ where the value $0.01$ is arbitrarily chosen. Then, we interpolate $\rho$ on a $P3$ finite element space on this new mesh and compute the now well-defined eigenvalue problem
\begin{equation}
  \begin{cases}
    -\div (\rho \nabla u) = \mu_k(\rho )\rho u \mbox { in } \{\rho >0.01\},\\
     \frac{\partial u}{\partial n} =0 \mbox { on } \partial  \{\rho >0.01\}
  \end{cases}
\end{equation}
to obtain precise reliable values of the eigenmodes.

\bigskip
\noindent{\bf Numerical Results.}
\begin{figure}
    \centering
    \includegraphics[width=0.24\textwidth]{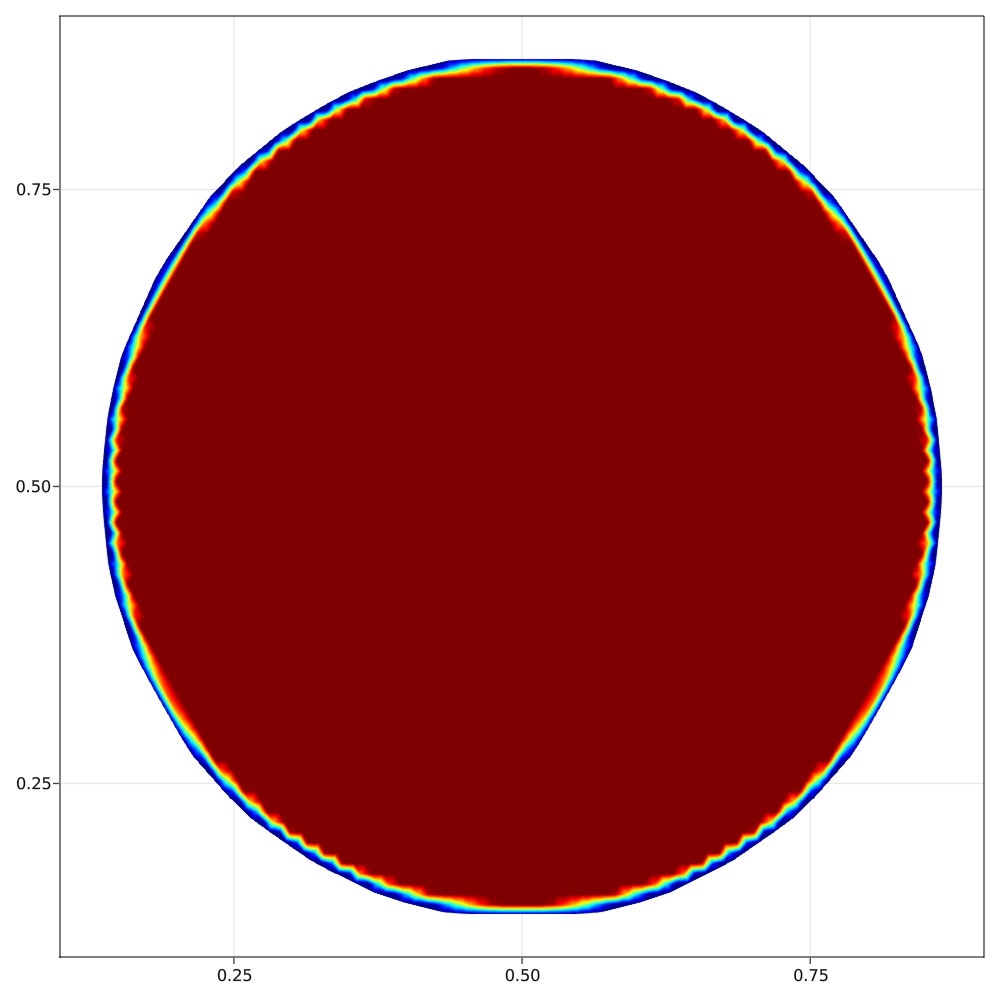}
    \includegraphics[width=0.24\textwidth]{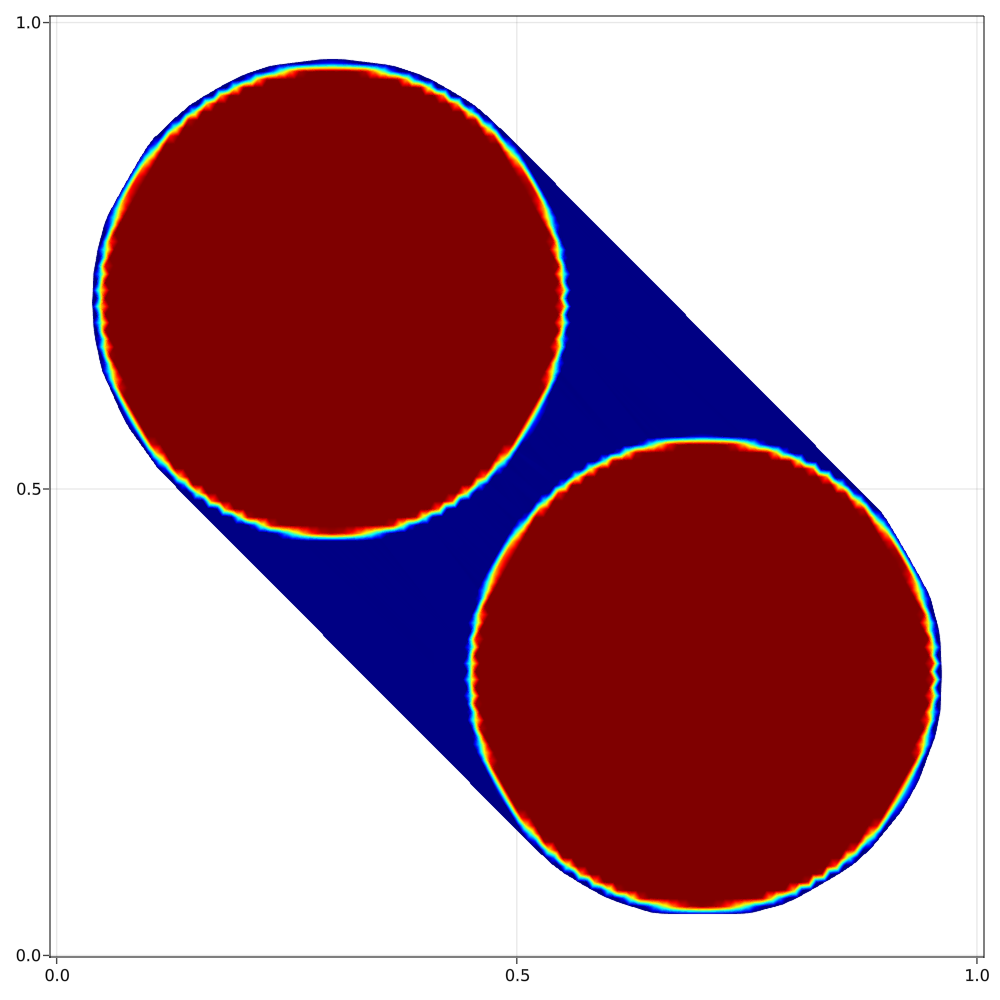}
    \includegraphics[width=0.24\textwidth]{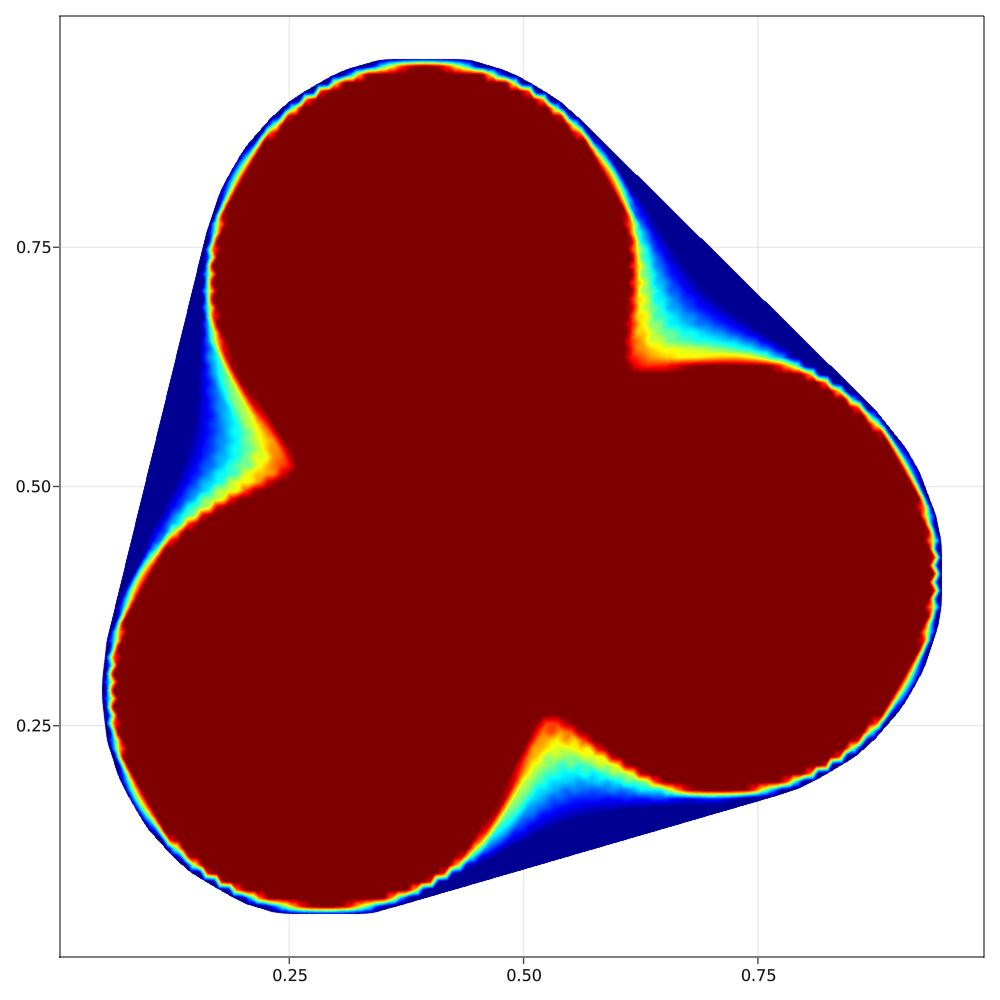}
    \includegraphics[width=0.24\textwidth]{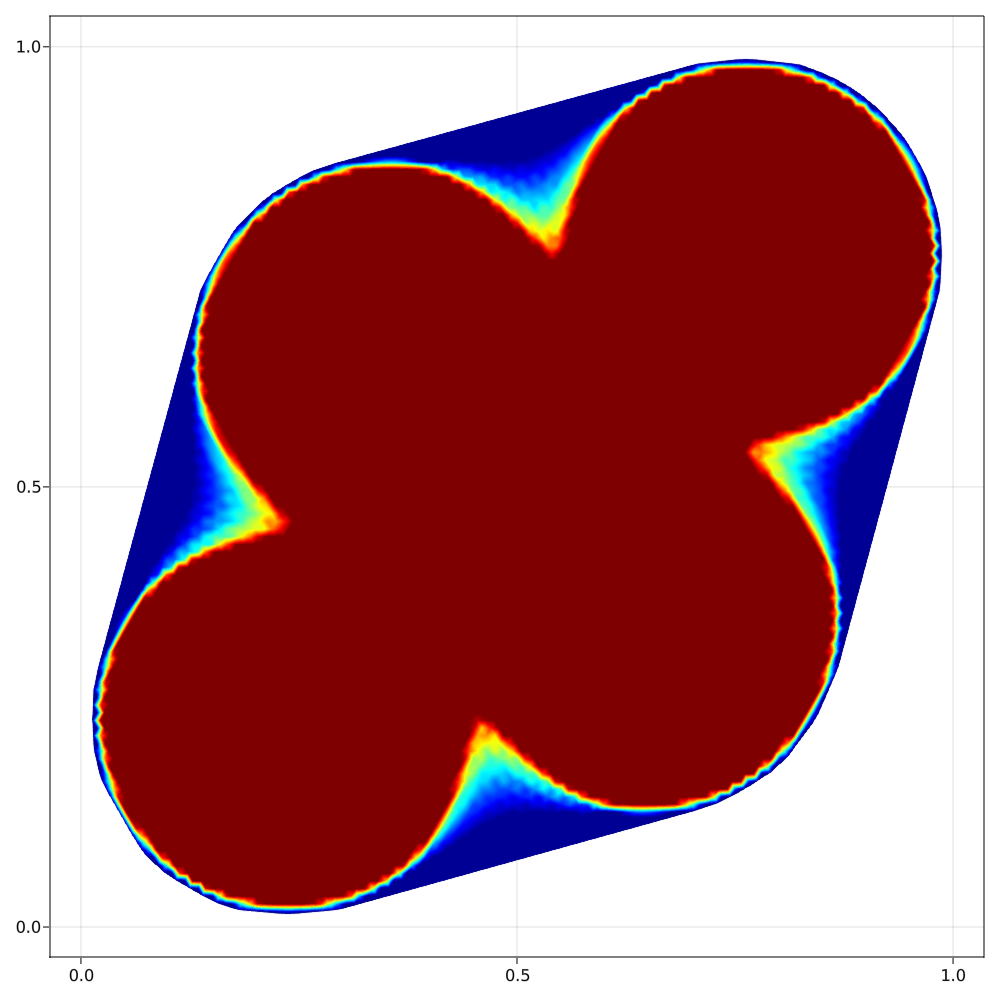}
    \includegraphics[width=0.24\textwidth]{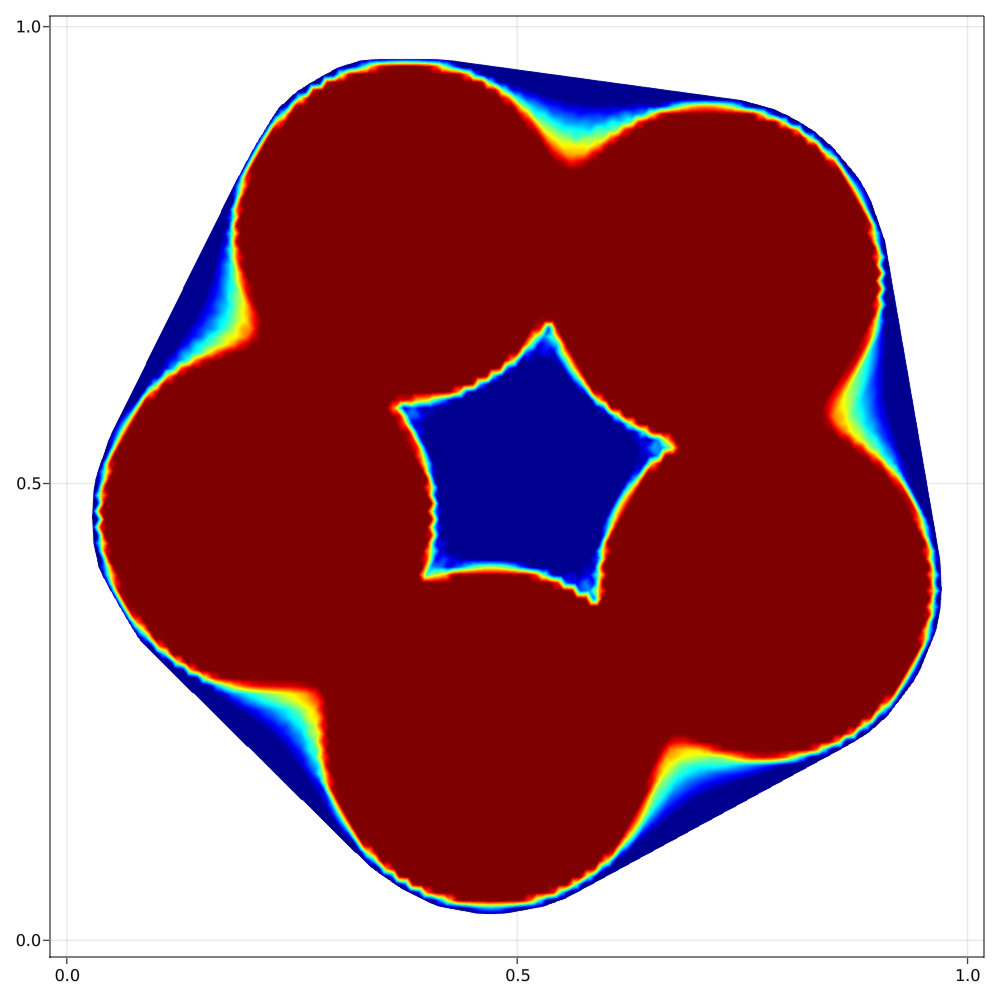}
    \includegraphics[width=0.24\textwidth]{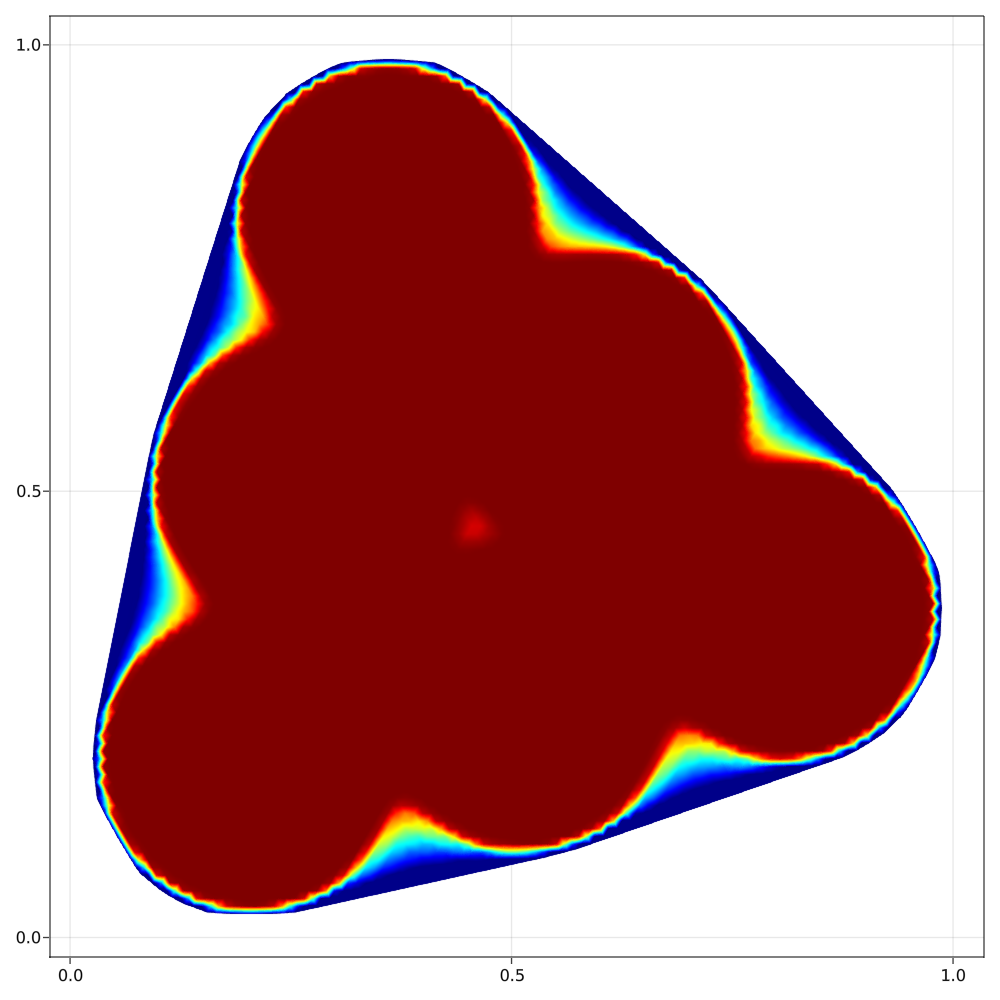}
    \includegraphics[width=0.24\textwidth]{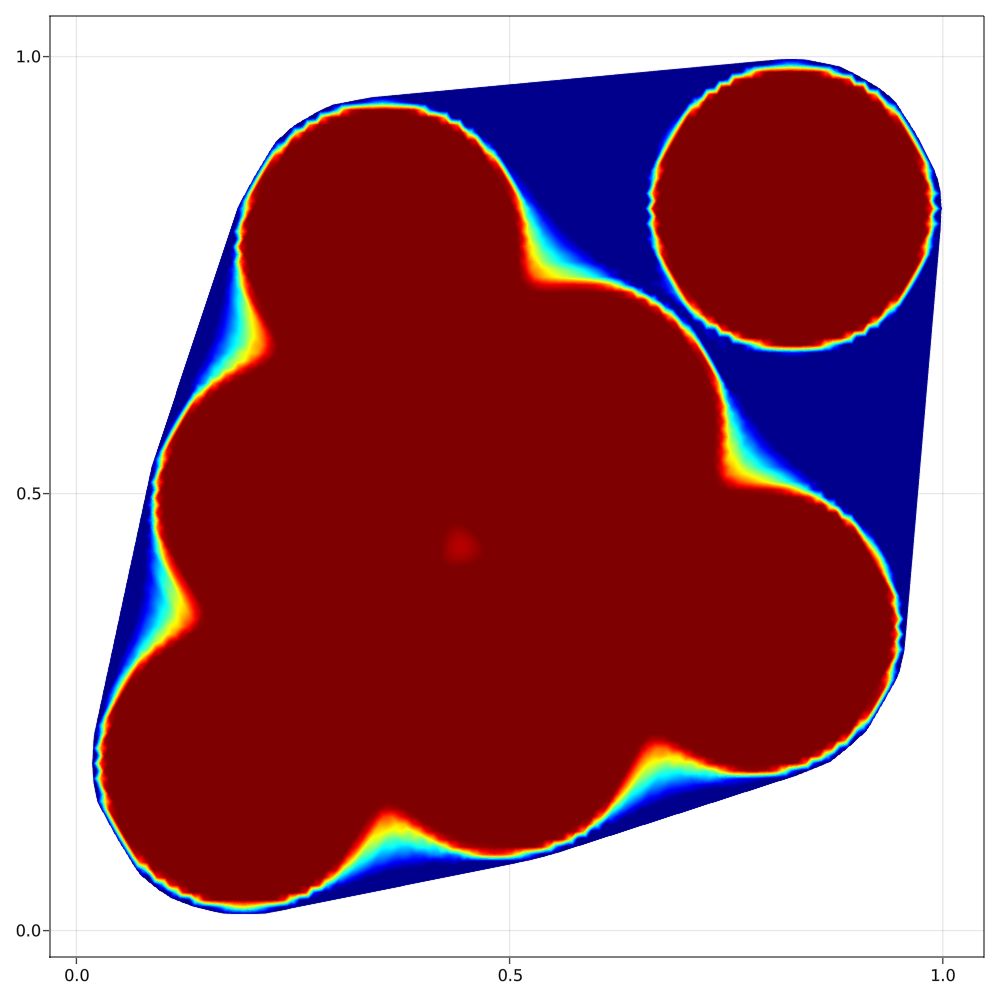}
    \includegraphics[width=0.24\textwidth]{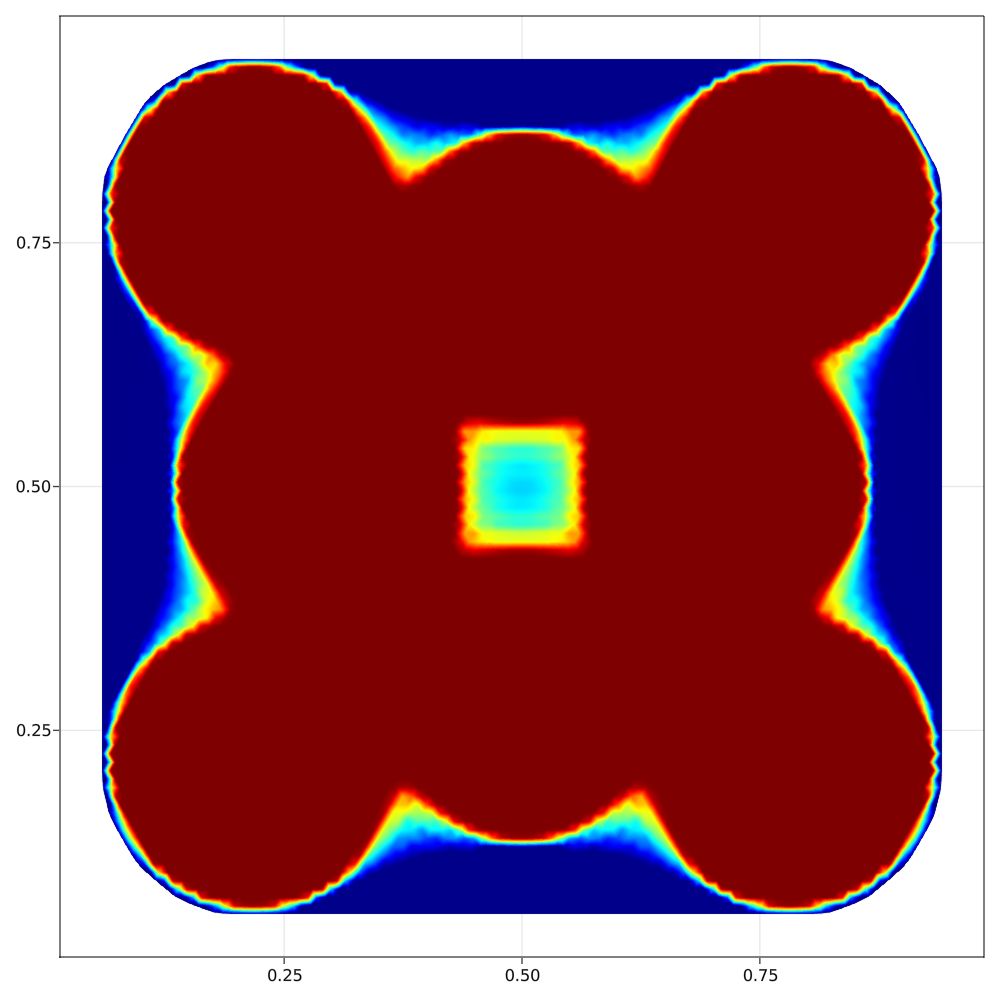}
    \caption{Approximation of the first eight optimal densities}
    \label{fig:numerical_results}
\end{figure}
Figure ~\ref{fig:numerical_results} displays the numerical results for $k=1,...,8$ plotted on the convex hull of $\{\rho > 0.01\}$. For $k=1$ and $k=2$, we obtained  respectively one and two disks, which meets the theory. For $k\geq3$,  optimal densities looks like homogenized union of disks.

In Table ~\ref{fig:comparison} we display the numerical values obtained by our  approach (for densities) compared both to the optimal values of \cite{AF12} and to the ones of  disjoint union of discs. We can observe that $k=4,k=5$ and $k=8$ are fairly different from the result in \cite{AF12}. Moreover, for $k=5$ and $k=8$, the "optimal" shape that one may somehow extract from an optimal density seems not to be a union of {\it simply connected} domains going beyond the framework in which the computation of \cite{AF12} have been carried out.

\begin{table}
  \begin{tabular}{|l||l|l|l|l|l|l|l|l|l|}
   \hline
   & $\mu_1$ & $\mu_2$ & $\mu_3$ & $\mu_4$ & $\mu_5$ & $\mu_6$ & $\mu_7$ & $\mu_8$ \\
  \hline
  Multiplicity & 2 & 2 & 3 & 3 & 3 & 4 & 4 & 4\\
  \hline
 Optimal densities & 10.65 & 21.28 & 32.92 & 43.90 & 54.47 & 67.25 & 77.96 & 89.47 \\
  \hline
  Optimal shapes, Antunes-Freitas \cite{AF12}&  &  & 32.79 & 43.43 & 54.08 & 67.04 & 77.68 & 89.22   \\
 \hline
 Union of discs & 10.65 & 21.30 & 31.95 & 42.60 & 53.25 & 63.90 & 74.55 & 88.85  \\
 \hline
  \end{tabular}
  \smallskip
\caption{Values comparison}
 \label{fig:comparison}
\end{table}

\bibliographystyle{mybst}
%\bibliography{References}

\end{document}